\documentclass[draft]{amsart}

\newtheorem{theorem}{Theorem}[section]
\newtheorem{lemma}[theorem]{Lemma}

\theoremstyle{definition}

\theoremstyle{remark}
\newtheorem{remark}[theorem]{Remark}

\numberwithin{equation}{section}

\usepackage{bbm}
\usepackage{bm}
\usepackage{color}
\usepackage{cancel}
\usepackage{amsrefs}
\usepackage{float}
\usepackage{enumerate}
\newcommand{\Bk}{\color{black}}

\newcommand{\bfX}{\bm X}
 
\newcommand{\vertiii}[1]{{\left\vert\kern-0.25ex\left\vert\kern-0.25ex\left\vert #1 
    \right\vert\kern-0.25ex\right\vert\kern-0.25ex\right\vert}}%

\usepackage{tabularx}                                            
  \newcolumntype{R}{>{\raggedleft\arraybackslash}X}
  \newcolumntype{L}{>{\raggedright\arraybackslash}X}
  \newcolumntype{C}{>{\centering\arraybackslash}X}
\usepackage{booktabs}

\newcommand{\DG}{\mathrm{DG}}
\newcommand{\mesh}{\mathcal{E}}
\newcommand{\aelip}{a_{\mathrm{ellip}}}

\newcommand{\err}{\bm{\tilde{e}}_h^n}

\usepackage{cleveref}    
\begin{document}

\title[Discontinuous Galerkin pressure correction scheme ]{A discontinuous Galerkin pressure correction scheme for the incompressible Navier--Stokes equations: Stability and convergence} 
\author{Rami Masri}
\address{Department of Computational and Applied Mathematics, Rice University, Houston, Texas 77005}
\curraddr{Department of Computational and Applied Mathematics, Rice University, Houston, Texas 77005}
\email{rami.masri@rice.edu}
\author{Chen Liu}
\address{Department of Mathematics, Purdue University, West Lafayette, Indiana 47907}
\curraddr{Department of Mathematics, Purdue University, West Lafayette, Indiana 47907}
\email{liu3373@purdue.edu}
\thanks{}

\author{Beatrice Riviere}
\address{Department of Computational and Applied Mathematics, Rice University, Houston, Texas 77005}
\curraddr{Department of Computational and Applied Mathematics, Rice University, Houston, Texas 77005}
\email{riviere@rice.edu}
\thanks{The third author is partially supported by NSF-DMS 1913291, NSF-DMS 2111459.}


\date{\today}


\begin{abstract}
A discontinuous Galerkin pressure correction numerical method for solving the incompressible Navier--Stokes equations is formulated and analyzed. We prove unconditional stability of the proposed scheme. Convergence of the discrete velocity is established by deriving a priori error estimates. Numerical results verify the convergence rates.
\end{abstract}
\maketitle
\section{Introduction}

The numerical simulation of the incompressible Navier-Stokes equations presents a challenging computational task primarily because of two  reasons: (a) the coupling of the velocity and pressure by the incompressibility constraint and (b) the nonlinearity of the convection term \cite{girault2005splitting,guermond2006overview}. The development of splitting schemes aims to overcome these difficulties by decoupling the nonlinearity in the  convection term from the pressure term. For an overview of such methods, we refer to the works of Glowinski \cite{Glowinski2003} and of Guermond, Minev, and Shen \cite{guermond2006overview}. In this paper, we will focus on pressure correction schemes. The basic idea of a non-incremental pressure correction scheme in time was first proposed by Chorin and Temam \cite{chorin1968numerical,temam1969approximation}. This scheme was subsequently modified by several mathematicians leading to two major variations: (1) the incremental scheme where a previous value of the pressure gradient is added \cite{goda1979multistep,van1986second} and (2) the rotational scheme where the non-physical boundary condition for the pressure is corrected by using the rotational form of the Laplacian \cite{timmermans1996approximate}. 

The main contribution of our work is the theoretical analysis of a discontinuous Galerkin (dG) discretization of the pressure correction approach. We derive stability and a priori error bounds on a family of regular meshes. The discrete velocities are approximated by discontinuous piecewise polynomials of degree $k_1$ and the discrete potential and pressure by polynomials of degree $k_2$. Stability of the solutions is obtained under the constraint $k_1-1\leq k_2 \leq k_1+1$ whereas the convergence of the scheme is obtained for the case $k_2 = k_1-1$ because of approximation properties. The proofs are technical and rely on several tools including special lift operators.

The semi-discrete error analysis of pressure correction schemes has been extensively studied, see for example the work by Shen and Guermond \cite{shen1996error,guermond2004error}. The use of finite element approximations in conjunction with pressure correction schemes is also well studied and a priori error estimates are established. Without being exhaustive, we refer to the work by Guermond and Quartapelle \cite{guermond1998approximation,guermond1998stability}, and to the work by Nochetto and Pyo \cite{nochetto2005gauge} for the analysis of finite element methods.

More recently, the combination of dG spatial approximations with pressure correction formulation in time to solve the incompressible Navier-Stokes equations has been the subject of several computational papers. These dG methods have multiple important features including local mass conservation and high order convergence.  In what follows, we mention a non-exhaustive list of such publications. Botti and Di Pietro employ a dG approximation to the velocity and a continuous Galerkin approximation to the pressure \cite{botti2011pressure}. 
 Liu et al. formulate an interior penalty dG method with the pressure correction approach for the computational estimation of rock permeability \cite{liu2019interior}. Piatkowski et al.  use a modified upwind scheme based on Vijayasundaram numerical flux, postprocess the projected velocity so that it is discretely divergence free, and perform numerical experiments of large scale 3D problems \cite{piatkowski2018stable}. We also mention the work by Fehn et al. where the stability of pressure correction and velocity correction dG methods is numerically investigated for small time step sizes, and the robustness of the methods is demonstrated for laminar high Reynolds number flow problems \cite{fehn2017stability}. 

Error analysis of dG methods for the \textit{steady state} Navier-Stokes equations has been established, see the work by Girault, Riviere and Wheeler \cite{girault2005discontinuous}. The authors also prove a priori error estimates for the time dependent Navier-Stokes equations where an operator splitting scheme is combined with dG methods in \cite{girault2005splitting}. To the best of our knowledge, the a  priori error analysis for a pressure correction dG approximation of the incompressible Navier-Stokes equations is missing from the literature.  

The outline of this paper is as follows. In Section \ref{sec:model_problem}, the model problem and the splitting scheme in time are presented. Section \ref{sec:dg_forms} introduces the dG forms and summarizes their properties. We present the fully discrete scheme and show existence and uniqueness of solutions in Section \ref{sec:numerical_scheme}. Unconditional stability is established in Section \ref{sec:stability}, and convergence of the discrete velocity is shown in Section~\ref{sec:convergence_velocity}. Conclusions follow. 

\section{Model problem and time discretization}\label{sec:model_problem}
Consider the incompressible time dependent Navier-Stokes equations with homogenous Dirichlet boundary condition for the velocity and a zero average constraint for the pressure. 
\begin{alignat}{2}
\partial_t \bm{u}  - \mu \Delta \bm{u}+ \bm{u}\cdot \nabla \bm{u}+ \nabla p &= \bm{f},  && \quad\mathrm{ in} \,\, \Omega \times (0,T], \label{eq:first_eq_NS}\\ 
\nabla \cdot \bm{u} & = 0, &&\quad \mathrm{in} \,\, \Omega \times (0,T], \\  
\bm{u} & = \bm{u}^0,  &&\quad \mathrm{in} \,\, \Omega \times \{0\}, \\ 
\bm{u} & = \bm{0},&& \quad \mathrm{on} \,\, \partial \Omega \times (0,T], 
\label{eq:dirichlet_bc}   \\ 
\int_{\Omega} p & =0, && \quad \forall t \in (0,T]. \label{eq:zero_avg}
\end{alignat}
Here, $\bm{u}$ is the fluid velocity, $p$ is the pressure, $\bm{f}$ is the external force, $\Omega$ is an open bounded polyhedral domain in $\mathbbm{R}^d$ where $d \in \{2,3\}$, and $\mu$ is the positive constant viscosity. Let $\tau$ denote the time step size and consider a uniform partition of the time interval $(0,T]$ into $N_T$ subintervals. Throughout the paper, we use the notation $\bm{g}^n = \bm{g}(t^n)$ and $q^n = q(t^n)$ for  given functions $\bm{g}$ and $q$ evaluated at $t^n=n\tau$. We now formulate a variation of the consistent splitting scheme that splits the operators and introduces additional velocity $\bm{v}^n$ and potential $\phi^n$ \cite{guermond2006overview,liu2019interior}. Initially, set $p^0 = 0$. For $n=1,\ldots, N_T$, given $\bm{u}^{n-1}$ and $p^{n-1}$, compute an intermediate velocity $\bm{v}^{n}$ such that 
\begin{alignat}{2}
    \bm{v}^{n}  - \tau \mu \Delta \bm{v}^{n} +\tau \bm{u}^{n-1}\cdot \nabla \bm{v}^{n} +\tau \nabla p^{n-1} &= \bm{u}^{n-1} + \tau \bm{f}^{n}, & & \quad \mathrm{ in} \, \Omega,  \label{eq:time_interm_vel} \\ 
    \bm{v}^n & = \bm{0}, & & \quad \mathrm{ on} \, \partial \Omega.
\end{alignat}
Given $\bm{v}^n$, compute $\phi^n$ such that, 
\begin{alignat}{2}
-\Delta \phi^n &= -\frac{1}{\tau}\nabla \cdot \bm{v}^n, && \quad \mathrm{ in} \, \Omega, \\ 
\nabla \phi^n \cdot \bm{n}& = 0, && \quad \mathrm{ on} \, \partial \Omega, \\
\int_\Omega \phi^n &=0.&&\Bk
\end{alignat} 
Given $(\bm{v}^n, p^{n-1}, \phi^n)$, update the pressure $p^n$ and the velocity $\bm{u}^n$ as follows. 
\begin{align}
p^n &= p^{n-1} + \phi^n - \delta \mu \nabla \cdot \bm{v}^n, \label{eq:time_update_press}\\ 
\bm{u}^n & = \bm{v}^n - \tau \nabla \phi^n. \label{eq:time_updat_vel} 
\end{align}
In \eqref{eq:time_update_press}, $\delta$ is a positive parameter that can be chosen in the interval $]0,1/(4d)]$,
as stated in the stability and convergence results proved in this work. The vector $\bm{n}$ is the unit outward normal vector
to $\partial\Omega$. 
\section{Notation and discontinuous Galerkin forms} \label{sec:dg_forms}
For a given non-negative integer $m$ and real number $r \geq 1$, define the Sobolev space on a domain $\mathcal{O} \subset \mathbbm{R}^d$: 
\begin{equation}
    W^{m,r}(\mathcal{O}) = \{ w \in L^r(\mathcal{O}):\quad \forall |\bm{\alpha}| \leq m, \quad D^{\bm{\alpha}} w \in L^r(\mathcal{O}) \}, 
\end{equation} 
where $\bm{\alpha}$ is a multi-index and $D^{\bm{\alpha}} w$ is the corresponding weak partial derivative of $w$. The usual Sobolev norms and semi-norms are denoted by  $ \| \cdot \|_{W^{m,r}(\mathcal{O})} $ and $| \cdot |_{W^{m,r}(\mathcal{O})}$ respectively. If $r=2$, we denote $H^m(\mathcal{O}) = W^{m,2}(\mathcal{O})$,  $\| \cdot \|_{H^m(\mathcal{O})} = \|\cdot\|_{W^{m,2}(\mathcal{O})}$ and $|\cdot |_{H^m(\mathcal{O})} = |\cdot |_{W^{m,2}(\mathcal{O})}$. 
The $L^2$ inner-product over $\Omega$ is denoted by $(\cdot,\cdot)$ and the resulting $L^2$ norm by $\Vert \cdot\Vert$.
Let $\mesh_h = \{E_k\}$ denote a family of regular partitions of the domain  $\Omega$ \cite{ciarlet2002finite}. That is, there exists a constant $\rho$ independent of $h$ such that
\begin{equation} \frac{h_E}{\rho_E} \leq \rho, \quad \forall E \in \mesh_h, \label{eq:mesh_regular}\end{equation}
where $h_E   = \mathrm{diam}(E)$ and $\rho_E$ is the maximum diameter of a ball inscribed in $E$. 
Let $h$ denote the maximum diameter of the elements in $\mesh_h$. 
We define the following broken Sobolev spaces.
\begin{alignat}{3}
    \bm{X} &= \{\bm{v}  \in L^2(\Omega)^d: && \quad \forall E \in \mesh_h,&& \quad \bm{v} \vert_E \in W^{2,2d/(d+1)}(E)^d  \}, \\ 
    M & = \{ q \in L^2(\Omega): && \quad \forall E \in \mesh_h, && \quad q\vert_E \in W^{1,2d/(d+1)} (E)\}.
\end{alignat}  
With this choice of spaces, we have that the trace of $\bm{v} \in \bm{X}$ and of each component of its gradient belongs to $L^2(\partial E)$ for all $E \in \mesh_h$. Similarly, the trace of $q \in M$ belongs to $L^2(\partial E)$ for all $E \in \mesh_h$. To see that this holds, we refer to Theorem 5.36 in \cite{adams2003sobolev}. In addition, we have that if $\Omega$ is a Lipschitz polygon and  $\bm{f} \in L^2(0,T; L^{2d/(d+1)}(\Omega)^d) $, then $(\bm{u},p) \in L^2(0,T; W^{2,2d/(d+1)}(\Omega)^d)  \times L^2(0,T;W^{1,2d/(d+1)}(\Omega))$ \cite{grisvard24elliptic,girault2005discontinuous,girault2012finite}. 

Denote by $\Gamma_h$ the set of all interior faces of the subdivision $\mesh_h$.  For an interior face $e \in \Gamma_h$, we associate a normal $\bm{n}_e$ and we denote 
by $E_e^1$ and $E_e^2$ the two elements that share $e$, such that $\bm{n}_e$ points from $E_e^1$ to $E_e^2$. 
Define the average and jump for a function $\bm{\theta} \in \bfX$ as such, 
\begin{align}
\{ \bm{\theta} \}  = \frac12 (\bm{\theta}|_{E_e^1} + \bm{\theta}|_{E_e^2}), \quad 
[\bm{\theta}]   = \bm{\theta}|_{E_e^1} - \bm{\theta}|_{E_e^2}, \quad \forall e = \partial E_e^1 \cap \partial E_e^2. 
\end{align}
For a boundary face, $e \in \partial \Omega$, the vector $ \bm{n}_e$ is chosen as the unit outward vector to $\partial \Omega$. The definition of the average and jump in this case are extended as such, 
\begin{equation} 
\{ \bm{\theta} \} = [\bm{\theta}] = \bm{\theta}|_{E_e}, \quad \forall e = \partial E_e \cap \partial \Omega. 
\end{equation}
Similar definitions are used for scalar valued functions, $q \in M$. We now introduce the dG formulations for the spatial discretization of \eqref{eq:time_interm_vel}-\eqref{eq:time_updat_vel}. For the convection term in \eqref{eq:time_interm_vel}, we use the same discretization form, $a_\mathcal{C}$,  as in \cite{girault2005discontinuous}.  To define this form, 
 we use the following notation: the vector $\bm{n}_E$ denotes the outward normal to $E$, the trace of a function $\bm{v}$
on the boundary of $E$ coming from the interior (resp. exterior) of $E$ is denoted by $\bm{v}^\mathrm{int}$ (resp. $\bm{v}^\mathrm{ext}$). By convention, $\bm{v}^\mathrm{ext}|_e = {\bf 0}$ if $e$ is a boundary face $(e\subset\partial\Omega)$. 
We also introduce the notation for the inflow boundary of $E$ with respect to a function $\bm{z}\in\bfX$: 
\begin{equation}
    \partial E_{-}^{\bm{z}} = \{ \bm{x} \in \partial E: \{\bm{z}  (\bm{x}) \}\cdot \bm{n}_E < 0 \}.
    \label{eq:inflow_boundary}
\end{equation}
With this notation, we define for $\bm{z}, \bm{w}, \bm{v}, \bm{\theta} \in \bfX$, 
\begin{multline}
a_\mathcal{C}(\bm{z};\bm{w}, \bm{v}, \bm{\theta}) = \sum_{E \in \mesh_h} \left( \int_E (\bm{w} \cdot \nabla \bm{v})\cdot \bm{\theta}  + \frac12 \int_E (\nabla \cdot \bm{w}) \, \bm{v} \cdot \bm{\theta} \right) \\ 
- \frac{1}{2} \sum_{e\in\Gamma_h \cup \partial \Omega} \int_{e} [\bm{w}] \cdot \bm{n}_e \{ \bm{v} \cdot \bm{\theta} \} + \sum_{E \in \mesh_h} \int_{\partial E_{-}^{\bm{z}}} | \{\bm{w}\} \cdot \bm{n}_E| (\bm{v}^{\mathrm{int}} - \bm{v}^{\mathrm{ext}}) \cdot  \bm{\theta}^{\mathrm{int}}.
\end{multline}
We recall the positivity property satisfied by $a_\mathcal{C}$ (see (1.18) in \cite{girault2005discontinuous}):
\begin{equation}\label{eq:cpositivity}
a_\mathcal{C}(\bm{w};\bm{w};\bm{v},\bm{v}) \geq 0, \quad \forall \bm{w}, \bm{v} \in \bfX.
\end{equation}
In the analysis, it will be useful to separate the upwind terms from the form $a_\mathcal{C}(\bm{w}; \bm{w},\bm{v},\bm{\theta})$. To this end, for $\bm{z}, \bm{w}, \bm{v}, \bm{\theta} \in \bm{X}$, we define:
\begin{multline}
\mathcal{C}(\bm{w}, \bm{v}, \bm{\theta}) =\sum_{E \in \mesh_h} \left( \int_E (\bm{w} \cdot \nabla \bm{v})\cdot \bm{\theta}  + \frac12 \int_E (\nabla \cdot \bm{w}) \, \bm{v} \cdot \bm{\theta} \right) 
\\ - \frac{1}{2} \sum_{e\in\Gamma_h \cup \partial \Omega} \int_{e} [\bm{w}] \cdot \bm{n}_e \{ \bm{v} \cdot \bm{\theta} \} .
\end{multline}
We also define the upwind terms
\begin{equation}
    \mathcal{U}(\bm{z}; \bm{w}, \bm{v}, \bm{\theta}) = \sum_{E \in \mesh_h } \int_{\partial E_{-}^{\bm{z}}} \{ \bm{w}\} \cdot \bm{n}_E(\bm{v}^{\mathrm{int}} - \bm{v}^{\mathrm{ext}})\cdot \bm{\theta}^{\mathrm{int}}. \label{eq:split_ac1}
\end{equation}
It follows that: 
\begin{equation}
 a_{\mathcal{C}}(\bm{w}; \bm{w}, \bm{v}, \bm{\theta}) = \mathcal{C}(\bm{w},\bm{v},\bm{\theta}) - \mathcal{U}(\bm{w};\bm{w},\bm{v},\bm{\theta}).   \label{eq:split_ac2}
\end{equation}
The discretization for the elliptic operator, $-\Delta \bm{v}$, is given as follows \cite{riviere2008discontinuous}. For $\bm{v}, \bm{\theta} \in \bfX$, 
\begin{multline}
a_{\epsilon} (\bm{v}, \bm{\theta}) = \sum_{E \in \mesh_h} \int_E \nabla \bm{v} : \nabla \bm{\theta} 
- \sum_{e\in\Gamma_h \cup \partial \Omega} \int_e \{\nabla \bm{v}\}\bm{n}_e \cdot [\bm{\theta}]  \\ 
+ \epsilon \sum_{e \in \Gamma_h \cup \partial \Omega} \int_e \{\nabla \bm{\theta} \}\bm{n}_e\cdot [\bm{v}] 
+  \sum_{e\in \Gamma_h \cup \partial \Omega} \frac{\sigma}{h_e} \int_{e} [\bm{v}]\cdot [\bm{\theta}].
\label{eq:eliptic_form_v}
\end{multline} 
In the above form, $h_e = |e|^{1/(d-1)}$, $\epsilon \in \{ -1,0,1\}$,  $\sigma > 0$ is a user specified penalty parameter.
The discretization for the term $ - \nabla p$ is given as follows. For $\bm{\theta} \in \bfX$ and $q \in M$, define  
\begin{equation}
b(\bm{\theta}, q) = \sum_{E \in \mesh_h }\int_E (\nabla \cdot \bm{\theta}) \, q  - \sum_{e \in \Gamma_h \cup \partial \Omega} \int_e \{q\}[\bm{\theta}]\cdot \bm{n}_e.  \label{eq:form_b} 
\end{equation}
To approximate $\bm{u}$ and $p$, we introduce discrete function spaces $\bm{X}_h \subset \bm{X}$ and  $M_{h0}\subset M_h \subset M $. For any integers $k_1\geq 1, k_2\geq 0$,
\begin{alignat}{2}
\bfX_h &= \{ \bm{v}_h \in (L^2(\Omega))^d:\quad &&\forall E \in \mesh_h,\quad \bm{v}_h \vert_{E} \in (\mathbbm{P}_{k_1}(E))^d\}, \\
M_h &= \{ q_h \in L^2(\Omega):\quad  && \forall E \in \mesh_h, 
\quad q_h \vert_E \in \mathbbm{P}_{k_2}(E)\}, \\
M_{h0}&=\{ q_h \in M_h: \quad && \int_\Omega q_h = 0\}. 
\end{alignat}
In the above, for $k \in \mathbbm{N}$, $\mathbbm{P}_k(E)$ denotes the space of polynomials of degree at most $k$. 
We will assume that (for reasons that will be evident in the analysis)
$$k_1 -1 \leq k_2 \leq k_1+1.  $$  
To discretize the elliptic operator $-\Delta \phi$, we define for $\phi_h, q_h\in M_h$,
\begin{multline}
    \aelip(\phi_h,q_h) = \sum_{E \in \mesh_h} \int_E \nabla \phi_h \cdot \nabla q_h - \sum_{e \in \Gamma_h} \int_e \{\nabla \phi_h \}\cdot \bm{n}_e [q_h] \\ 
   -  \sum_{e\in \Gamma_h} \int_e \{ \nabla q_h\}\cdot \bm{n}_e [\phi_h] +  \sum_{e\in \Gamma_h} \frac{\tilde{\sigma}}{h_e} \int_e [\phi_h] [q_h]. 
\end{multline}
Here, $\tilde{\sigma}>0$ is a penalty parameter. For $\bm{\theta}\in \bfX$, define the energy norm as follows: 
\begin{equation}
    \| \bm{\theta}\|^2_{\DG} = \sum_{E \in \mesh_h}  \|\nabla \bm{\theta}  \|^2_{L^2(E)} + \sum_{e \in \Gamma_h \cup \partial \Omega } \frac{\sigma}{h_e} \|[\bm{\theta}]\|^2_{L^2(e)}.
  \end{equation}
  For $q \in M$, the energy semi-norm is defined as such: 
  \begin{equation}
    \vert q_h \vert^2_{\DG} = \sum_{E \in \mesh_h}  \|\nabla q \|^2_{L^2(E)} 
+ \sum_{e \in \Gamma_h } \frac{\tilde{\sigma}}{h_e} \|[q]\|^2_{L^2(e)}.
  \end{equation}
 Clearly, $\vert \cdot \vert_{\DG}$ is a norm for the space $M_{h0}$. 
We recall the following coercivity properties.
\begin{alignat}{2}
    a_\epsilon (\bm{\theta}_h, \bm{\theta}_h) &\geq \kappa \|\bm{\theta}_h \|_{\DG}^2, && \quad  \forall \bm{\theta}_h \in \bm{X}_h, \label{eq:coercivity_a_epsilon} \\
    \aelip(q_h, q_h ) &\geq \frac12 | q_h |^2_{\DG}, &&  \quad \forall q_h \in M_h.  \label{eq:coercivity_a_ellip} 
\end{alignat}
It is shown that \eqref{eq:coercivity_a_epsilon} holds with $\kappa = 1$ if $\epsilon = 1$ and with $\kappa= 1/2$ if $\epsilon \in \{ -1, 0 \}$ and $\sigma$ is large enough \cite{riviere2008discontinuous}. Property \eqref{eq:coercivity_a_ellip} holds for $\tilde{\sigma}$ large enough \cite{riviere2008discontinuous}.  In what follows, we will assume that \eqref{eq:coercivity_a_ellip} and \eqref{eq:coercivity_a_epsilon} are satisfied. Further, we recall that $a_{\epsilon}$ is continuous on $\bm{X}_h \times \bm{X}_h$ with respect to  $\| \cdot \|_{\DG}$ and  $\aelip$ is continuous on $M_h \times M_h$ with respect to  $|\cdot |_{\DG}$ \cite{riviere2008discontinuous}. There exists a constant $C$ independent of $h$, $\bm{\theta}_h$, $\bm{z}_h$, $q_h$, and $\xi_h$ such that   
\begin{alignat}{2}
a_\epsilon(\bm{\theta}_h, \bm{z}_h) &\leq C \| \bm{\theta}_h \|_{\DG} \, \|\bm{z}_h\|_{\DG}, && \quad \forall \bm{\theta}_h, \bm{z}_h \in \bm{X}_h, \label{eq:continuity_a_epsilon}\\ 
\aelip(q_h, \xi_h) & \leq C | q_h |_{\DG} \, |\xi_h |_{\DG},  && \quad \forall q_h, \xi_h \in M_h  \label{eq:continuity_aellip}. 
\end{alignat}
We now state an equivalent expression for the form $b$.   
\begin{lemma}
We have the following equivalent form for $b(\bm{\theta}, q)$, 
\begin{equation}
    b(\bm{\theta}, q ) = -\sum_{E \in \mesh_h} \int_E \bm{\theta} \cdot \nabla q  + \sum_{e\in \Gamma_h} \int_e \{ \bm{\theta} \}\cdot \bm{n}_e [q], \quad \forall (\bm{\theta},q) \in \bm{X} \times M.  \label{eq:equiv_form_b}
\end{equation} 
\end{lemma}
\begin{proof}
The result is obtained by applying Green's theorem to the first term in \eqref{eq:form_b} and by using the identity
\begin{eqnarray}
    \sum_{e \in \Gamma_h} \int_{e} [\bm{\theta}\cdot \bm{n}_e q ] = \sum_{e\in \Gamma_h} \int_{e} \{\bm{\theta} \}\cdot \bm{n}_e [q] + \sum_{e \in \Gamma_h} \int_e \{q\} [\bm{\theta}]\cdot \bm{n}_e .
\end{eqnarray}
\end{proof} 
We will also make use of lift operators, which are useful tools in the theoretical analysis of dG methods.  
Given $e \in \Gamma_h \cup \partial \Omega$, we introduce a new lift operator $r_e: (L^2(e))^d \rightarrow M_h$ as follows 
\begin{equation}
\int_{\Omega} r_e(\bm{\zeta}) q_h  = \int_e \{q_h\} \bm{\zeta} \cdot \bm{n}_e, \quad \forall q_h \in M_h.   \label{eq:def_re}
\end{equation}
We next recall the lift operator $\bm{g}_e$ introduced in \cite{di2010discrete}: given an interior face $e\in \Gamma_h$, the operator  $\bm{g}_e: L^2(e) \rightarrow \bfX_h$ satisfies  
\begin{equation}
\int_{\Omega} \bm{g}_e(\zeta) \cdot \bm{\theta}_h = \int_e \{ \bm{\theta}_h \} \cdot \bm{n}_e \zeta, \quad \forall \bm{\theta}_h \in \bfX_h. \label{eq:def_ge}
\end{equation}
With these definitions, we  construct two operators, $R_h: \bfX_h\rightarrow M_h$ and $\bm{G}_h: M_h\rightarrow \bfX_h$ 
\begin{alignat}{2}
    R_h([\bm{\theta}_h]) &= \sum_{e\in \Gamma_h \cup \partial \Omega} r_e([\bm{\theta}_h]),&& \quad \bm{\theta}_h \in \bfX_h,   \label{eq:lift_op_1}\\ 
    \bm{G}_h([\beta_h]) &= \sum_{e\in \Gamma_h} \bm{g}_e([\beta_h]),&& \quad \beta_h \in M_h. \label{eq:lift_op_2} 
\end{alignat}
It is easy to check that $R_h$ and $\bm{G}_h$ are linear operators. We next show boundedness of the lift operators. 
\begin{lemma}\label{lemma:properties_lift_operators}
    There exist constants $M_{k_2}, \tilde{M}_{k_1} > 0$ independent of $h$ but depending on the polynomial degrees $k_2$ and $k_1$ 
respectively, such that the following bounds hold: 
       \begin{align}
       \|R_h([\bm{\theta}_h])\|&\leq M_{k_2} \left(\sum_{e \in \Gamma_h \cup \partial \Omega} h_e^{-1}\|[\bm{\theta}_h]\|_{L^2(e)}^2\right)^{1/2}, \quad \forall \bm{\theta}_h \in \bfX_h, \label{eq:lift_prop_r} \\ 
       \|\bm{G}_h([q_h])\| &\leq \tilde{M}_{k_1} \left(\sum_{e \in \Gamma_h} h_e^{-1}\|[q_h]\|_{L^2(e)}^2\right)^{1/2}, \quad\forall q_h\in M_h. \label{eq:lift_prop_g}
       \end{align}
   \end{lemma}
   \begin{proof}
We will show \eqref{eq:lift_prop_r}. We use the definitions of $R_h$ \eqref{eq:lift_op_1} and of $r_e$ \eqref{eq:def_re}. We have 
\begin{align*}
& \|R_h([\bm{\theta}_h]) \|^2 = \int_{\Omega} R_h([\bm{\theta}_h])R_h([\bm{\theta}_h])  
= \sum_{e\in\Gamma_h \cup \partial \Omega} \int_{\Omega} r_e([\bm{\theta}_h]) R_h([\bm{\theta}_h]) \\ 
 & = \sum_{e\in \Gamma_h } \int_{e} \frac{1}{2} \left(R_h([\bm{\theta}_h])\vert_{E_e^1} + R_h([\bm{\theta}_h])\vert_{E_e^2}\right)[\bm{\theta}_h]\cdot \bm{n}_e + \sum_{e\in \partial \Omega} \int_e R_h([\bm{\theta}_h])\vert_{E_e} [\bm{\theta}_h]\cdot \bm{n}_e.  
           \end{align*}
           Cauchy-Schwarz's inequality yields
           \begin{align*}
            & \|R_h([\bm{\theta}_h]) \|^2 \leq \frac{1}{2}\left( \sum_{e\in \Gamma_h} h_{E_e^1} \|R_h([\bm{\theta}_h]) \vert_{E_e^1}\|^2_{L^2(e)} \right)^{1/2} \left(\sum_{e\in \Gamma_h} h_{E_e^1}^{-1} \|[\bm{\theta}_h ]\cdot \bm{n}_e\|^2_{L^2(e)}\right)^{1/2} \\ & + \frac{1}{2}\left( \sum_{e\in \Gamma_h} h_{E_e^2} \|R_h([\bm{\theta}_h]) \vert_{E_e^2}\|^2_{L^2(e)} \right)^{1/2} \left(\sum_{e\in \Gamma_h} h_{E_e^2}^{-1} \|[\bm{\theta}_h ]\cdot \bm{n}_e\|^2_{L^2(e)}\right)^{1/2}  \\ 
            &+ \left( \sum_{e\in \partial \Omega} h_{E_e} \|R_h([\bm{\theta}_h]) \vert_{E_e}\|^2_{L^2(e)} \right)^{1/2} \left(\sum_{e\in \partial \Omega} h_{E_e}^{-1} \|[\bm{\theta}_h] \cdot \bm{n}_e\|^2_{L^2(e)}\right)^{1/2}.
           \end{align*}
           With a local trace inequality and the fact that $h_{E} \geq h_e$ for any element $E$ having a face $e$, we obtain the result  with $M_{k_2}$ depending on a trace constant. The proof of \eqref{eq:lift_prop_g} follows a similar argument and is omitted for brevity.
   \end{proof}
Further, with the definitions of the lift operators \eqref{eq:lift_op_1} and \eqref{eq:lift_op_2}, we have the following equivalent forms to \eqref{eq:form_b} and \eqref{eq:equiv_form_b} respectively.
\begin{align}
    b(\bm{\theta}_h, q_h) & = ( \nabla_h   \cdot \bm{\theta}_h, q_h ) - (R_h([\bm{\theta}_h]), q_h), 
\quad \forall \bm{\theta}_h\in\bfX_h, \, \forall q_h \in M_h, \label{eq:def_b_lift} \\ 
    b(\bm{\theta}_h, q_h) & = -(\nabla_h q_h, \bm{\theta}_h) + (\bm{G}_h([q_h]), \bm{\theta}_h),
\quad \forall \bm{\theta}_h\in\bfX_h, \, \forall q_h \in M_h, \label{eq:def_b_lift_2}
 \end{align}
where $\nabla_h$ and $\nabla_h \cdot$ denote the broken gradient and divergence operators respectively.
\section{Numerical scheme}\label{sec:numerical_scheme}
We start with setting $p_h^0 = \phi_h^0 = 0$. We let $\bm{u}_h^0$ be the local $L^2$ projection of $\bm{u}^0$ onto $\bm{X}_h$. 
\begin{equation}
\int_{E} (\bm{u}_h^0 - \bm{u}^0) \cdot \bm{\theta}_h = 0, \quad \forall \bm{\theta}_h \in  (\mathbbm{P}_{k_1}(E))^d, \quad\forall E\in\mathcal{E}_h.
\end{equation} 
For $n =1, \ldots, N_T$, given $(\bm{u}_h^{n-1}, p_h^{n-1}) \in \bfX_h \times  M_h$ compute $\bm{v}_h^n \in \bfX_h$ such that for all $\bm{\theta}_h \in \bfX_h$, 
\begin{multline}
    (\bm{v}_h^n, \bm{\theta}_h) + \tau a_{\mathcal{C}}(\bm{u}_h^{n-1};\bm{u}_h^{n-1}, \bm{v}_h^n, \bm{\theta}_h) + \tau \mu a_\epsilon(\bm{v}_h^n, \bm{\theta}_h ) = (\bm{u}_h^{n-1}, \bm{\theta}_h) \\ + \tau b(\bm{\theta}_h,p_h^{n-1}) + \tau (\bm{f}^{n}, \bm{\theta}_h). \label{eq:intermidiate_velocity}
\end{multline}
Next, compute $\phi_h^n \in  M_{h0} $ such that for all $q_h \in M_{h0}$,
\begin{equation}
a_{\mathrm{ellip}}(\phi_h^n, q_h) = - \frac{1}{\tau} b(\bm{v}_h^n, q_h). \label{eq:pressure_correction}
\end{equation} 
Finally, compute $p_h^n \in M_h$ and  $\bm{u}_h^n \in \bfX_h$ such that for all $q_h \in M_h$ and for all  $\bm{\theta}_h \in \bfX_h$, 
\begin{align}
    (p_h^n, q_h) & = (p_h^{n-1}, q_h) + (\phi^n_h, q_h) - \delta \mu \, b(\bm{v}_h^n, q_h), \label{eq:update_pressure_1} \\ 
    (\bm{u}_h^n, \bm{\theta}_h) & = (\bm{v}_h^n, \bm{\theta}_h) + \tau b(\bm{\theta}_h, \phi_h^n).\label{eq:update_velocity_1}
    \end{align}
    Using \eqref{eq:def_b_lift} and \eqref{eq:def_b_lift_2},  steps \eqref{eq:update_pressure_1} and \eqref{eq:update_velocity_1} can be written as follows. For all $q_h \in M_h$ and $\bm{\theta}_h \in \bm{X}_h$, 
\begin{align}
(p_h^n, q_h) & = (p_h^{n-1}, q_h) + (\phi^n_h, q_h) - \delta \mu (\nabla_h \cdot \bm{v}_h^n - R_h([\bm{v}_h^n]), q_h), \label{eq:update_pressure} \\ 
(\bm{u}_h^n, \bm{\theta}_h) & = (\bm{v}_h^n, \bm{\theta}_h) - \tau(\nabla_h \phi_h^n- \bm{G}_h([\phi^n_h]), \bm{\theta}_h).\label{eq:update_velocity}
\end{align}
    \begin{lemma}\label{prep:pressure_in_Mh0} 
        Let $n \geq 0$ and $p_h^n \in M_h$ be defined by \eqref{eq:update_pressure}, then $p_h^n \in M_{h0}$. 
    \end{lemma}
\begin{proof}
We present a proof by induction on $n$. For $n=0$, the statement trivially holds. Assume $p_h^{n-1} \in M_{h0}$. Let $q_h = 1$ in \eqref{eq:update_pressure_1} and use \eqref{eq:def_b_lift_2}. We obtain:
\begin{equation}
    \int_{\Omega} p_h^n = \int_\Omega p_h^{n-1} + \int_{\Omega} \phi_h^n. 
 \end{equation} 
We conclude the result by using the induction hypothesis and the fact that $\phi_h^n \in M_{h0}$.
\end{proof}
\begin{lemma}
    Given $(\bm{v}_h^{n-1},\bm{u}_h^{n-1}, p_h^{n-1}) \in \bm{X}_h \times \bm{X}_h \times M_{h0}$, there exists a unique solution $(\bm{v}_h^n, \bm{u}_h^{n}, p_h^n) \in \bm{X}_h \times \bm{X}_h \times M_{h0}$ to the dG scheme given by \eqref{eq:intermidiate_velocity} - \eqref{eq:update_velocity}. 
\end{lemma} \begin{proof}
\textit{Existence of the intermediate velocity} $\bm{v}_h^n$: Since this is a linear problem in finite dimension, it suffices to show uniqueness of the solution. Suppose there exist two solutions $\bm{v}_h^n$ and $\bm{\tilde{v}}_h^n$ to \eqref{eq:intermidiate_velocity} and let 
$\bm{\chi}^n_h$ denote the difference between the two solutions, 
\begin{equation}
    \bm{\chi}^n_h = \bm{v}_h^n - \bm{\tilde{v}}_h^n.  \nonumber
\end{equation}
Then,  recalling that $a_{\mathcal{C}}$ is linear in the third argument, and choosing $\bm{\theta}_h = \bm{\chi}_h^n$, we obtain 
\begin{equation}
(\bm{\chi}_h^n , \bm{\chi}_h^n)  + \tau a_\mathcal{C}(\bm{u}_h^{n-1}; \bm{u}_h^{n-1}, \bm{\chi}_h^n, \bm{\chi}_h^n) + \tau \mu a_\epsilon(\bm{\chi}_h^n, \bm{\chi}_h^n) = 0.  \nonumber
\end{equation} 
With the positivity property of $a_\mathcal{C}$ \eqref{eq:cpositivity} and the coercivity of $a_\epsilon$ \eqref{eq:coercivity_a_epsilon}, we conclude that  \begin{equation}
    \| \bm{\chi}_h^n \|^2 + \kappa \tau \mu  \| \bm{\chi}_h^n \|^2_{\DG} \leq 0.
\end{equation}
This implies that $\bm{\chi}_h^n=\bm{0}$. Thus, the solution to \eqref{eq:intermidiate_velocity} is unique. 

\textit{Existence of} $\bm{u}_h^n$ \textit{and} $p_h^n$: The existence of $\phi_h^n \in M_{h0}$ follows by similar arguments, the coercivity property of $a_{\mathrm{ellip}}(\cdot, \cdot)$, and the fact that $| \cdot |_{\DG}$ is a norm for the space $M_{h0}$. The existence of $p_h^n \in M_{h0}$ follows directly from \eqref{eq:update_pressure} and Lemma \ref{prep:pressure_in_Mh0}. The existence of  $\bm{u}_h^n \in \bm{X}_h$ follows from \eqref{eq:update_velocity}. 
\end{proof}
\section{Stability}\label{sec:stability}
In this section, we present a stability analysis of the proposed scheme. 
We proceed by noting the following identity. 
\begin{lemma}\label{prep:weak_divergence}
   Let $(\bm{u}^n_h, \bm{v}^n_h, \phi_h^n) \in \bm{X}_h \times \bm{X}_h \times M_{h0}$ be defined by algorithm \eqref{eq:intermidiate_velocity}-\eqref{eq:update_velocity}. For all $q_h \in M_h$ and $n \geq 1$, the following holds.
   \begin{multline}
    b(\bm{u}_h^n, q_h) = b(\bm{v}_h^n, q_h) + \tau a_{\mathrm{ellip}}(\phi_h^n, q_h) \\ - \tau \sum_{e\in \Gamma_h} \frac{\tilde{\sigma}}{h_e}  \int_{e} [\phi^n_h][q_h] + \tau (\bm{G}_h([\phi^n_h]), \bm{G}_h([q_h])).\label{eq:preposition} 
   \end{multline}
   In addition, we have 
    \begin{align}    
        b(\bm{u}_h^n, q_h) &= - \tau \sum_{e\in \Gamma_h} \frac{\tilde{\sigma}}{h_e} \int_{e} [\phi^n_h][q_h] + \tau (\bm{G}_h([\phi^n_h]), \bm{G}_h([q_h])). \label{eq:preposition_1}
    \end{align}
\end{lemma}
\begin{proof}
    Let $q_h \in M_{h}$ be arbitrary but fixed. Take $\bm{\theta}_h = \nabla_h q_h$ in \eqref{eq:update_velocity}. We have, 
    \begin{align}
    (\bm{u}_h^n, \nabla_h q_h )   &= (\bm{v}_h^n, \nabla_h q_h ) - \tau  (\nabla_h \phi_h^n-\bm{G}_h([\phi^n_h]), \nabla_h q_h). \label{eq:attempt2_1}  
    \end{align}
    Next, take $\bm{\theta}_h = \bm{G}_h([q_h])$ in \eqref{eq:update_velocity}.  We have, 
    \begin{align}
    (\bm{u}_h^n, \bm{G}_h([q_h]))   &= (\bm{v}_h^n, \bm{G}_h([q_h])) - \tau (\nabla_h \phi_h^n - \bm{G}_h([\phi_h^n]), \bm{G}_h([q_h])). \label{eq:attempt2_2}
    \end{align}
    Subtracting \eqref{eq:attempt2_1} from \eqref{eq:attempt2_2} and using \eqref{eq:def_b_lift_2}, we obtain 
    \begin{align}
    b(\bm{u}_h^n , q_h) & = b(\bm{v}_h^n , q_h )  + \tau (\nabla_h \phi_h^n - \bm{G}_h([\phi_h^n]), \nabla_h q_h - \bm{G}_h([q_h])). \label{eq:attempt2_3}
    \end{align}
Observe that 
\begin{multline}
  (\nabla_h \phi_h^n - \bm{G}_h([\phi_h^n]), \nabla_h q_h - \bm{G}_h([q_h])) =   a_{\mathrm{ellip}}(\phi_h^n, q_h) \\ - \sum_{e\in\Gamma_h} \frac{\tilde{\sigma}}{h_e} \int_e [\phi^n_h][q_h] +  (\bm{G}_h([\phi^n_h]), \bm{G}_h([q_h])). \label{eq:attempt2_4}
\end{multline}
Substituting \eqref{eq:attempt2_4} in \eqref{eq:attempt2_3} implies \eqref{eq:preposition}.
 To see that \eqref{eq:preposition_1} holds, define $\beta_h \in M_{h0}$ as $\beta_h = q_h - \langle q_h\rangle$ with $ \langle q_h \rangle = (1/|\Omega|)\int_{\Omega} q_h$. Since $b(\bm{v}_h^n, \langle q_h \rangle ) = \aelip(\phi_h^n , \langle q_h \rangle ) = 0$, we obtain 
 \begin{equation} \nonumber
    b(\bm{u}_h^n , q_h) = b(\bm{v}_h^n,\beta_h) + \tau a_{\mathrm{ellip}}(\phi_h^n, \beta_h) - \tau \sum_{e\in\Gamma_h} \frac{\tilde{\sigma}}{h_e} \int_e [\phi^n_h][q_h]  + \tau  (\bm{G}_h([\phi^n_h]), \bm{G}_h([q_h])). 
 \end{equation}
 By applying \eqref{eq:pressure_correction}, we obtain \eqref{eq:preposition_1}. 
\end{proof}
We recall the following $L^q$ bound for the broken Sobolev spaces. Since $\bm{X} \subset H^1(\mesh_h)^d$, the proof of this lemma can be found in
Lemma 6.2 in \cite{girault2005discontinuous} in 2D and in Corollary 2.2 in 3D in \cite{lasis2003poincare}.
\begin{lemma}\label{lemma:Poincare_inequality}
There exists a constant $C_P$ independent of $h$ and $\tau$ but depending on $q$ such that 
\begin{equation}
    \| \bm{\theta} \|_{L^{q}(\Omega)} \leq C_P \| \bm{\theta} \|_{\DG}, \quad \forall \bm{\theta} \in \bm{X}, \label{eq:discreter_poincare}
\end{equation}
where $2\leq q<\infty$ in 2D $(d=2)$ and $2\leq q\leq 6$ in 3D $(d=3)$.  
\end{lemma}
To state and prove our stability result, for $n \geq 0$, we introduce auxiliary functions $S_h^n \in M_h$ and $ \xi_h^n \in M_{h}$. 
\begin{alignat}{2}
S_h^0 = 0, & \quad  S_h^n = \delta \mu \sum_{i=1}^{n} \big( \nabla_h \cdot \bm{v}_h^i - R_h([\bm{v}_h^i])\big), && \quad n \geq 1, \label{eq:def_S}\\ 
\xi_h^0 = 0, & \quad \xi_h^n = p_h^n + S_h^n, && \quad n \geq 1.  \label{eq:def_xi}
\end{alignat} 
Note that $S_h^n \in M_{h0}$ since with the equivalent definition \eqref{eq:def_b_lift}, we have 
\begin{equation}\label{eq:S_in_Mh0}
\int_{\Omega} S_h^n = \delta\mu \sum_{i = 1}^n b(\bm{v}_h^i, 1) = 0, \quad n \geq 1. 
\end{equation}
This implies that $\xi_h^n \in M_{h0}$ since $p_h^n \in M_{h0}$, see Lemma \ref{prep:pressure_in_Mh0}.
We are now ready to prove the main result of this section.
\begin{theorem} \label{prep:stability} 
Assume that $\sigma \geq M_{k_2}^2/d$, $\tilde{\sigma} \geq \tilde{M}_{k_1}^2$, $\delta \leq \kappa/(2d)$, and $\bm{u}^0 \in L^2(\Omega)^d$. Then, the pressure correction scheme \eqref{eq:intermidiate_velocity}-\eqref{eq:update_velocity} is unconditionally stable. For all $\tau > 0$ and  $1 \leq  m \leq N_T$,
\begin{multline}
    \| \bm{u}_h^m\|^2 +  \frac{\kappa  \mu }{2}  \tau \sum_{n=1}^{m} \| \bm{v}_h^n \|_{\DG}^2  
+   \frac{1}{2} \tau^2 \vert \xi_h^m \vert_{\DG}^2  + \frac{1}{\delta \mu}\tau \|S_h^m\|^2 \\ \leq \|\bm{u}^0 \|^2 +  \ \frac{2C_P^2}{\kappa \mu}  \tau \sum_{n=1}^{m} \|\bm{f}^{n}\|^2. 
\end{multline}  
\end{theorem}
\begin{proof}
Let $\bm{\theta}_h = \bm{v}_h^n$ in \eqref{eq:intermidiate_velocity}. 
We use the positivity property of $a_\mathcal{C}$ \eqref{eq:cpositivity} and  the coercivity of $a_{\epsilon}$\eqref{eq:coercivity_a_epsilon}. We obtain the following.
\begin{equation}
    \begin{split}
  \frac12  (\| \bm{v}_h^n\|^2 -   \|\bm{u}_h^{n-1} \|^2 +   \|\bm{v}_h^n - \bm{u}_h^{n-1} \|^2 )+ \kappa \mu \tau  & \| \bm{v}_h^n \|^2_{\DG} \\ & \leq \tau b(\bm{v}_h^n, p_h^{n-1}) + \tau (\bm{f}^{n}, \bm{v}_h^n). 
    \end{split}  
  \label{eq:stability_dg_1}
\end{equation}
Next, let $\bm{\theta}_h = \bm{u}_h^n$ in \eqref{eq:update_velocity_1},
\begin{equation}
 \frac12 (\|\bm{u}_h^n\|^2 -  \|\bm{v}_h^n\|^2 +  \|\bm{u}_h^n - \bm{v}_h^{n} \|^2 )  = \tau b(\bm{u}_h^n, \phi_h^n) . \label{eq:looking_at_updatevel}     
\end{equation}
Let $\bm{\theta}_h = \bm{u}_h^n - \bm{v}_h^n$ in \eqref{eq:update_velocity_1}.  We obtain, 
\begin{align}
 \|\bm{u}_h^n - \bm{v}_h^n \|^2 & = \tau b(\bm{u}_h^n - \bm{v}_h^n, \phi_h^n).   \nonumber
\end{align}
Using \eqref{eq:preposition} in Lemma \ref{prep:weak_divergence}, we find
\begin{equation}
    \|\bm{u}_h^n - \bm{v}_h^n \|^2 = \tau^2 a_{\mathrm{ellip}}(\phi_h^n,\phi_h^n) - \tau^2 \sum_{e\in\Gamma_h} \frac{\tilde{\sigma}}{h_e}\|[\phi^n_h]\|^2_{L^2(e)} + \tau^2 \| \bm{G}_h([\phi^n_h]) \|^2. \nonumber 
\end{equation}
Similarly, using \eqref{eq:preposition_1} in Lemma \ref{prep:weak_divergence}, we obtain 
\begin{equation}
     b(\bm{u}_h^n, \phi_h^n) = -\tau \sum_{e\in\Gamma_h} \frac{\tilde{\sigma}}{h_e}\|[\phi^n_h]\|^2_{L^2(e)} + \tau \|\bm{G}_h([\phi_h^n])\|^2 . \nonumber
\end{equation}
Substituting the above expressions in \eqref{eq:looking_at_updatevel}, we obtain 
\begin{multline}
   \frac{1}{2} (\|\bm{u}_h^n\|^2 - \|\bm{v}_h^n\|^2) + \frac{\tau^2}{2}a_{\mathrm{ellip}}(\phi_h^n, \phi_h^n)\\  + \frac{\tau^2}{2} \sum_{e\in\Gamma_h}  \frac{\tilde{\sigma}}{h_e}\| [\phi^n_h]\|^2_{L^2(e)}   = \frac{\tau^2}{2} \| \bm{G}_h([\phi^n_h])\|^2 . \label{eq:using_prep_1}
\end{multline}
The last term is estimated using \eqref{eq:lift_prop_g}: 
\begin{equation}
    \|\bm{G}_h([\phi_h^n]) \| \leq \tilde{M}_{k_1} \left(\sum_{e\in \Gamma_h} h_e^{-1} \|[\phi_h^n]\|_{L^2(e)}^2 \right)^{1/2}.
\end{equation} 
Using the above estimate in \eqref{eq:using_prep_1} yields 
\begin{equation}
   \frac{1}{2}( \|\bm{u}_h^n\|^2 - \|\bm{v}_h^n\|^2) + \frac{\tau^2}{2} a_{\mathrm{ellip}}(\phi_h^n, \phi_h^n) + \frac{\tau^2}{2} \sum_{e\in\Gamma_h} (\tilde{\sigma} -\tilde{M}^2_{k_1}) h_e^{-1}\| [\phi^n_h]\|^2_{L^2(e)} \leq  0. 
\end{equation}
With the assumption on the penalty parameter, $\tilde{\sigma} \geq \tilde{M}_{k_1}^2$, we obtain 
\begin{equation}
    \frac12\|\bm{u}_h^n\|^2 - \frac12 \|\bm{v}_h^n\|^2 + \frac{\tau^2}{2}a_{\mathrm{ellip}}(\phi_h^n, \phi_h^n) \leq  0. \label{eq:error_eq_update_vel}
\end{equation}
We add \eqref{eq:error_eq_update_vel} to \eqref{eq:stability_dg_1}, 
\begin{multline}
  \frac12 ( \| \bm{u}_h^n\|^2 -   \|\bm{u}_h^{n-1} \|^2 +   \|\bm{v}_h^n - \bm{u}_h^{n-1} \|^2) + \kappa \mu \tau \| \bm{v}_h^n \|^2_{\DG}  + \frac{\tau^2}{2} a_{\mathrm{ellip}}(\phi_h^n,\phi_h^n) \\ \leq  \tau b(\bm{v}_h^n,p_h^{n-1}) + \tau (\bm{f}^{n}, \bm{v}_h^n).  \label{eq:error_after_adding}
\end{multline} 
To handle the first term on the right-hand side of the above inequality, we use the auxiliary function $S_h^n$ defined in \eqref{eq:def_S} and write
\begin{equation}
    b(\bm{v}_h^n, p_h^{n-1}) =  b(\bm{v}_h^n, p_h^{n-1}+S_h^{n-1})  -  b(\bm{v}_h^n, S_h^{n-1}). \label{eq:rhs_term}
\end{equation}
With the equivalent expression \eqref{eq:def_b_lift}, the second term in \eqref{eq:rhs_term} is rewritten:
\begin{multline}
b(\bm{v}_h^n,S_h^{n-1}) = \left(\nabla_h \cdot \bm{v}_h^n - R_h([\bm{v}_h^n]) , S_h^{n-1}\right)  \\  = \frac{1}{\delta \mu} (S_h^n - S_h^{n-1}, S_h^{n-1})  = \frac{1}{2\delta\mu} \left(\|S_h^n \|^2 - \|S_h^{n-1}\|^2 - \| S_h^n - S_h^{n-1}\|^2 \right).\label{eq:second_term_final_form}
\end{multline}
Recall the definition of $\xi_h^n$ in \eqref{eq:def_xi} and note that \eqref{eq:update_pressure} implies
\begin{equation}
\xi_h^n-\xi_h^{n-1} = \phi_h^n, \quad \forall n \geq 1. \label{eq:writing_pressure_projection_interm_gamma}
\end{equation}
Since $S_h^n, p_h^n \in M_{h0}$, we can apply \eqref{eq:pressure_correction}. With \eqref{eq:writing_pressure_projection_interm_gamma}, we have 
\begin{equation}
    b(\bm{v}_h^{n}, \xi_h^{n-1}) =- \tau a_{\mathrm{ellip}}(\phi_h^n, \xi_h^{n-1} ) = -\tau \aelip(\xi_h^n - \xi_h^{n-1}, \xi_h^{n-1}). \label{eq:first_term}
\end{equation} 
Since $a_{\mathrm{ellip}}(\cdot, \cdot)$ is symmetric, we have 
\begin{align}
   b(\bm{v}_h^n, \xi_h^{n-1}) &= -\frac{\tau}{2} \left(  a_{\mathrm{ellip}}(\xi_h^n, \xi_h^n ) -  a_{\mathrm{ellip}}(\xi_h^{n-1},\xi_h^{n-1}) -  a_{\mathrm{ellip}}(\xi_h^{n}-\xi_h^{n-1}, \xi_h^{n} - \xi_h^{n-1})\right) \nonumber \\ 
    & = -\frac{\tau}{2}\left(a_{\mathrm{ellip}}(\xi_h^n, \xi_h^n ) -  a_{\mathrm{ellip}}(\xi_h^{n-1},\xi_h^{n-1}) - a_{\mathrm{ellip}}(\phi_h^n, \phi_h^n) \right).\label{eq:first_term_final_form}
\end{align}
Substituting \eqref{eq:first_term_final_form} and \eqref{eq:second_term_final_form} in \eqref{eq:rhs_term} yields 
\begin{multline}
    b(\bm{v}_h^n, p_h^{n-1}) = -\frac{\tau}{2} \left( a_{\mathrm{ellip}}(\xi_h^n, \xi_h^n ) - a_{\mathrm{ellip}}(\xi_h^{n-1},\xi_h^{n-1}) - a_{\mathrm{ellip}}(\phi_h^n,\phi_h^n)\right) \\ - \frac{1}{2\delta \mu} \left(\|S_h^n \|^2 - \|S_h^{n-1}\|^2 - \| S_h^n - S_h^{n-1}\|^2 \right).
    \label{eq:rhs_final_form} 
\end{multline}
With the above expression, \eqref{eq:error_after_adding} becomes 
\begin{multline}
    \frac12 (\| \bm{u}_h^n\|^2 - \|\bm{u}_h^{n-1} \|^2 +   \|\bm{v}_h^n - \bm{u}_h^{n-1} \|^2 )+ \kappa \mu \tau \| \bm{v}_h^n \|_{\DG}^2  \\ +\frac{\tau^2}{2} \left( a_{\mathrm{ellip}}(\xi_h^n,\xi_h^{n}) - a_{\mathrm{ellip}}(\xi_h^{n-1}, \xi_h^{n-1}) \right)  +
    \frac{\tau}{2\delta \mu} \left(\|S_h^n \|^2 - \|S_h^{n-1}\|^2  \right) \\ 
    \leq \tau (\bm{f}^{n}, \bm{v}_h^n) + \frac{\tau}{2\delta \mu} \| S_h^n - S_h^{n-1}\|^2. \label{eq:stability_semi_final}
\end{multline}
To handle the second term on the right-hand side of \eqref{eq:stability_semi_final}, recall that from the definition of $S_h^n$, see \eqref{eq:def_S}, 
$$ S_h^{n} - S_h^{n-1} = \delta \mu( \nabla_h \cdot \bm{v}_h^n - R_h([\bm{v}_h^n])).$$ 
Using \eqref{eq:lift_prop_r}, the assumption that $\delta \leq \kappa/(2d)$ and $\sigma \geq M_{k_2}^2/d$, we have 
\begin{align}
   \frac{1}{2\delta \mu} \| S_h^n - S_h^{n-1}\|^2 &\leq  \delta \mu \| \nabla_h \cdot \bm{v}_h^n \|^2 +  \delta \mu \| R_h([\bm{v}_h^n])\|^2,  \nonumber\\
    & \leq d \delta \mu  \| \nabla_h \bm{v}_h^n\|^2 +  \delta \mu M^2_{k_2}  \sum_{e\in\Gamma_h \cup \partial \Omega} h_e^{-1}\|[\bm{v}_h^n]\|_{L^2(e)}^2, \nonumber \\ 
    & \leq  \frac{ \kappa  \mu }{2}\| \nabla_h \bm{v}_h^n \|^2 + \frac{\kappa \mu}{2}\sum_{e\in\Gamma_h \cup \partial \Omega} \sigma h_e^{-1} \|[\bm{v}_h^n]\|_{L^2(e)}^2.\label{eq:bounding_Shn_stability_proof} 
\end{align}
Substituting the above bound in \eqref{eq:stability_semi_final} yields, 
\begin{multline}
    \frac12( \| \bm{u}_h^n\|^2 - \|\bm{u}_h^{n-1} \|^2 +   \|\bm{v}_h^n - \bm{u}_h^{n-1} \|^2 )+ \frac{\kappa \mu }{2} \tau \| \bm{v}_h^n \|_{\DG}^2  \\ +\frac{\tau^2}{2} \left( a_{\mathrm{ellip}}(\xi_h^n,\xi_h^{n}) - a_{\mathrm{ellip}}(\xi_h^{n-1}, \xi_h^{n-1}) \right)  +
    \frac{\tau}{2\delta \mu} \left(\|S_h^n \|^2 - \|S_h^{n-1}\|^2  \right) 
    \leq \tau (\bm{f}^{n}, \bm{v}_h^n). \label{eq:stability_semi_final1}
\end{multline}
Using Cauchy-Schwarz's inequality, Poincare's inequality \eqref{eq:discreter_poincare},  and Young's inequality, we have 
\begin{equation}
   | (\bm{f}^{n}, \bm{v}_h^n)| \leq \frac{C_P^2}{\kappa \mu} \|\bm{f}^{n}\|^2 + \frac{\kappa\mu}{4}  \|\bm{v}_h^n\|_{\DG}^2.  \label{eq:bounding_fn_stability_proof}
\end{equation}
 We substitute the above abound in \eqref{eq:stability_semi_final1}, multiply by 2 and sum the resulting inequality from $n=1$ to $n=m$. We obtain  
\begin{multline}
    \| \bm{u}_h^m\|^2 + \frac{\kappa \mu}{2} \tau \sum_{n=1}^{m}  \|\bm{v}_h^n \|_{\DG}^2 + \tau^2 a_{\mathrm{ellip}}(\xi_h^m, \xi_h^m) + \frac{1}{\delta \mu}\tau \|S_h^m\|^2 \\\leq \|\bm{u}_h^0 \|^2 +  \frac{2C_P^2}{\kappa \mu}\tau\sum_{n=1}^{m} \|\bm{f}^{n}\|^2  + \tau^2 a_\mathrm{ellip}(\xi_h^0, \xi_h^0) + \frac{1}{\delta \mu}\tau \| S_h^{0}\|^2. 
\end{multline}
Recall that $\xi_h^0= S_h^0=  0$. With the coercivity property of $\aelip$ \eqref{eq:coercivity_a_ellip} and the stability of the $L^2$ projection, we attain the result. 
\end{proof}
\begin{remark} 
   The particular choice for the upper bound on $\delta$ in the statement of Theorem \ref{prep:stability} is chosen to simplify the writeup. However, it can be altered and a similar result would hold. For example, the condition on  $\delta$  can be chosen as  $\delta < \kappa/d$. This will lead to a larger constant $C$  resulting from Young's inequality in \eqref{eq:bounding_fn_stability_proof}.
\end{remark}

\section{Error estimates for the velocity} \label{sec:convergence_velocity}
In this section, we denote $k = k_1$ and consider the case  $k_2 = k-1$. 
To simplify notation, for any function $\bm{v} \in (W^{1,3}(\Omega) \cap L^{\infty}(\Omega))^d$, we define the norm: 
\begin{equation}
    \vertiii{\bm{v}} = \|\bm{v}\|_{L^{\infty}(\Omega)} + \vert \bm{v}\vert_{W^{1,3}(\Omega)}. 
\end{equation}
We assume that $\bm{u}^0
\in H^1_0(\Omega)^d$ and that $\nabla \cdot \bm{u}^0 = 0 $. 
Throughout the rest of this paper, we denote by  $C$ a constant independent of $\mu, h ,$ and $\tau$ and by $C_\mu$ a constant independent of $h$ and $\tau$ but depending on $e^{1/\mu}$. These constants may take different values when used in different places and may depend on certain norms of the true solution $(\bm{u}, p)$. These norms are associated with the spaces specified by the regularity assumptions of Theorem \ref{prep:first_error_estimate_velocity}. 
\subsection{Approximation}\label{subsec:interpolation}
We will make use of the operator $\Pi_h: H^1_0(\Omega)^d \rightarrow \bfX_h$  that preserves the
discrete divergence \cite{riviere2008discontinuous,chaabane2017convergence}:
\begin{equation}
b(\Pi_h \bm{u} (t), q_h) = b(\bm{u} (t) ,q_h)=0 , \quad \forall q_h \in M_h, \forall 0\leq t \leq T. \label{eq:def_pi_h}
\end{equation}
This operator satisfies the following approximation properties \cite{chaabane2017convergence}. For $E \in \mesh_h$, $1\leq r \leq \infty$, $1 \leq s \leq k+1$, $0\leq t \leq T$, and $\bm{u}(t) \in \bm{X} \cap (W^{s,r}(E) \cap H^1_0(\Omega))^d$,  
\begin{align}
\| \Pi_h \bm{u}(t) - \bm{u}(t) \|_{L^r(E)} &\leq Ch_E^s \vert \bm{u}(t) \vert_{W^{s,r}(\Delta_E)}, \label{eq:approximation_prop_local_1}\\
\| \nabla (\Pi_h \bm{u}(t) - \bm{u}(t)) \|_{L^r(E) }& \leq C h_E^{s-1}  \vert \bm{u}(t) \vert_{W^{s,r}( \Delta_E)}\label{eq:approximation_prop_local_2},   
\end{align}
where $\Delta_E$ is a macro element that contains $E$. For $ 0\leq t\leq T$, if $\bm{u}(t) \in (W^{s,r}(\Omega) \cap H_0^1(\Omega))^d$, then bounds \eqref{eq:approximation_prop_local_1} and \eqref{eq:approximation_prop_local_2} yield the global estimates: 
\begin{align}
    \| \Pi_h \bm{u}(t) - \bm{u}(t) \|_{L^r(\Omega)} &\leq Ch^s \vert \bm{u}(t) \vert_{W^{s,r}(\Omega)}, \label{eq:approximation_prop_1} \\ 
     \| \Pi_h \bm{u}(t) - \bm{u}(t) \|_{\mathrm{DG}}& \leq C h^{s-1}  \vert \bm{u}(t) \vert_{H^{s}(\Omega)} \label{eq:approximation_prop_2}.
    \end{align}
In the subsequent sections, the following stability bound for the operator $\Pi_h $ will be used. 
    \begin{lemma}
Fix $t \in [0,T]$. Assume that $\bm{u}(t) \in (L^{\infty}(\Omega)\cap W^{1,3}(\Omega)\cap H^1_0(\Omega))^d$. Then, we have the following bound. 
\begin{equation}
    \|\Pi_h \bm{u}(t)\|_{L^{\infty}(\Omega)} \leq C \vertiii{\bm{u}(t)}.\label{eq:stability_linf_pih}
\end{equation}
\end{lemma}
\begin{proof}
Let $\mathcal{I}_h \bm{u}(t) \in \bm
{X}_h$ be a Lagrange interpolant of $\bm{u}(t)$. Fix $E\in \mesh_h$. With Minkowski's inequality, an inverse estimate, stability and approximation properties of the Lagrange interpolant and \eqref{eq:approximation_prop_local_1}, we have 
\begin{align*}
\| \Pi_h \bm{u}(t)\|_{L^{\infty}(E)} &\leq \| \Pi_h \bm{u}(t) - \mathcal{I}_h \bm{u}(t)\|_{L^{\infty}(E)} + \| \mathcal{I}_h \bm{u}(t)\|_{L^{\infty}(E)}, \\ 
& \leq Ch_{E}^{-d/3} \|\Pi_h \bm{u}(t) - \mathcal{I}_h \bm{u}(t) \|_{L^3(E)} + \|\bm{u}(t)\|_{L^{\infty}(E)}, \\ 
& \leq Ch_{E}^{1-d/3} |\bm{u}(t)|_{W^{1,3}(\Delta_E)} + \| \bm{u}(t)\|_{L^{\infty}(E)}, \\ 
& \leq C \vertiii{\bm{u}(t)}.\end{align*}
Given that $\|\Pi_h \bm{u}(t)\|_{L^{\infty}(\Omega)} = \max_{E \in \mesh_h} \| \Pi_h \bm{u}(t)\|_{L^{\infty}(E)}$ and the above bound is uniform in $E$, we obtain the result. 
\end{proof}
Define the local $L^2$ projection  $\pi_h: L^2(\Omega) \rightarrow M_h$ as follows. For $ 0\leq t\leq T$,  a given function $p(t) \in L^2(\Omega)$,  and any $E \in \mesh_h$,  
\begin{equation}
    \int_{E} (\pi_h p(t) - p(t))q_h = 0, \quad  \forall q_h \in \mathbbm{P}_{k-1}(E) .
\end{equation}
For $0\leq t \leq T$ and $p(t) \in H^s(\Omega)$, the following estimate holds. For all $E \in \mesh_h$, 
\begin{align} \label{eq:l2_proj_approximation}
    \|\pi_h p(t) - p(t) \|_{L^2(E)} + h_E \|\nabla_h(\pi_h p(t) - p(t) )\|_{L^2(E)} \leq Ch_E^{\min(k,s)}\vert p(t) \vert_{H^s(E)}. 
\end{align}
We also recall that the local $L^2$ projection is stable in the dG norm \cite{GiraultLiRiviere2016}. 
\begin{lemma}\label{lemma:stability_l2_projection}
Fix $t \in [0,T]$. Assume that $p(t) \in H^1(\Omega)$. Then, 
\begin{equation}
    | \pi_h p(t)|_{\DG} \leq C | p(t) |_{H^1(\Omega)}.\label{eq:bd_phN}
\end{equation} 
\end{lemma}
\subsection{Error equations}
We recall that the dG discretization is consistent, in the sense that the true solution $(\bm{u}, p)$  of \eqref{eq:first_eq_NS}-\eqref{eq:zero_avg} satisfies for $n \geq 1 $
\begin{multline}
    ((\partial_t \bm{u})^n, \bm{\theta}_h) + a_\mathcal{C}(\bm{u}^n;\bm{u}^n, \bm{u}^n, \bm{\theta}_h) + \mu a_{\epsilon} (\bm{u}^n, \bm{\theta}_h) \\ = b(\bm{\theta}_h, p^n) + (\bm{f}^n, \bm{\theta}_h), \quad \forall \bm{\theta}_h \in \bm{X}_h.  \label{eq:true_sol_first}
\end{multline}
For readability, we define the following discretization errors, $\bm{\tilde{e}}_h^{n}\in\bm{X}_h, \, \bm{e}_h^{n}\in\bm{X}_h$:
\begin{align}
\bm{\tilde{e}}_h^{n}= \bm{v}_h^n - \Pi_h \bm{u}^n, \quad \bm{e}_h^{n} = \bm{u}_h^n - \Pi_h \bm{u}^{n},  \quad \forall n\geq 0.  \label{eq:def_error_functions}
\end{align}
In the above, we set $\bm{v}_h^0 = \bm{u}_h^0$. Multiplying \eqref{eq:true_sol_first} by $\tau$ and subtracting it from \eqref{eq:intermidiate_velocity}, we have for all $n \geq 1$ 
\begin{multline}
    (\err, \bm{\theta}_h) + \tau a_\mathcal{C}(\bm{u}_h^{n-1};\bm{u}_h^{n-1}, \err, \bm{\theta}_h) +  \tau R_{\mathcal{C}} (\bm{\theta}_h) + \tau \mu a_{\epsilon}(\err, \bm{\theta}_h)\\  = (\bm{e}_h^{n-1}, \bm{\theta}_h) - \tau \mu a_{\epsilon}(\Pi_h \bm{u}^n - \bm{u}^n, \bm{\theta}_h )
  + \tau  b(\bm{\theta}_h, p_h^{n-1} - p^n) + R_t(\bm{\theta}_h), \quad  \forall \bm{\theta}_h \in \bm{X}_h.  \label{eq:first_err_eq}
\end{multline} 
In the above, we set 
\begin{align}
R_{\mathcal{C}} (\bm{\theta}_h) & = a_\mathcal{C}(\bm{u}_h^{n-1}; \bm{u}_h^{n-1}, \Pi_h \bm{u}^n, \bm{\theta}_h) -  a_\mathcal{C}(\bm{u}^{n}; \bm{u}^{n}, \bm{u}^n, \bm{\theta}_h), \label{eq:defRC}\\ 
R_{t}(\bm{\theta}_h) & = \tau((\partial_t \bm{u})^n, \bm{\theta}_h) + (\Pi_h \bm{u}^{n-1} - \Pi_h \bm{u}^{n}, \bm{\theta}_h).\label{eq:defRt}
\end{align}
With \eqref{eq:def_pi_h} and \eqref{eq:def_b_lift}, we have for all $n\geq 1$ 
\begin{equation}
b(\Pi_h \bm{u}^n,q_h) = (\nabla_h \cdot \Pi_h \bm{u}^n - R_h([\Pi_h \bm{u}^n]),q_h) = 0, \quad \forall q_h \in M_h. \label{eq:equivalent_def_Pi_h}
\end{equation}
Using \eqref{eq:equivalent_def_Pi_h} and \eqref{eq:pressure_correction}, we obtain for all $n\geq 1$
\begin{equation}
\aelip(\phi_h^n, q_h) = -\frac{1}{\tau} b(\bm{\tilde{e}}_h^n, q_h), \quad \forall q_h \in M_{h0}. 
\label{eq:error_pressure_correction}
\end{equation}
Inserting $\Pi_h \bm{u}^n$ in \eqref{eq:update_velocity} yields for all $n \geq 1$ 
\begin{equation}
    (\bm{e}_h^n, \bm{\theta}_h) = (\bm{\tilde{e}}_h^n, \bm{\theta}_h) -  \tau(\nabla_h \phi_h^n- \bm{G}_h([\phi^n_h]), \bm{\theta}_h), \quad \forall  \bm{\theta}_h \in \bm{X}_h. \label{eq:sec_err_eq}
\end{equation}
\subsection{Intermediate results }We proceed by showing some intermediate properties of the error functions. 
\begin{lemma}\label{prep:weak_div_error}  
Fix $q_h\in M_h$. The following holds for $n\geq 1$,
    \begin{multline}
        b(\bm{e}_h^n, q_h) = b(\err, q_h) + \tau\aelip(\phi_h^n,q_h)  \\ -\tau \sum_{e \in \Gamma_h} \frac{\tilde{\sigma}}{h_e} \int_e [\phi_h^n][q_h] + \tau(\bm{G}_h([\phi_h^n]), \bm{G}_h([q_h])).   \label{eq:div_error} 
    \end{multline}
    Further, for $ n \geq 1$, 
    \begin{align}
        b(\bm{e}_h^n, q_h) & = -\tau \sum_{e \in \Gamma_h} \frac{\tilde{\sigma}}{h_e} \int_e [\phi_h^n][q_h] + \tau(\bm{G}_h([\phi_h^n]), \bm{G}_h([q_h])). \label{eq:div_error_2}
    \end{align}
 In addition, for $n \geq 1$,
    \begin{align}
     \|\err - \bm{e}_h^{n-1}\|^2 =&  \|\bm{e}_h^{n} - \bm{e}_h^{n-1} \|^2 + \tau^2(\| \nabla_h \phi_h^n \|^2 + \|\bm{G}_h([\phi_h^n])\|^2) 
+ \tau^2 (A^n_1-A^n_2) \nonumber\\ 
&  +  \tau^2\left( \sum_{e\in\Gamma_h} \frac{\tilde{\sigma}}{h_e}\|[\phi_h^n - \phi_h^{n-1}]\|^2_{L^2(e)} - \|\bm{G}_h([\phi_h^n - \phi_h^{n-1}])\|^2\right)\nonumber \\
&   -2\tau^2(\nabla_h \phi_h^n, \bm{G}_h([\phi_h^n])) + 2\delta_{n,1}\tau b(\bm{e}_h^{0} , \phi_h^1) , \label{eq:error_diff}
    \end{align}
 where $\delta_{n,1}$ is the Kronecker delta and $A^n_1$ and $A^n_2$ for $n \geq 1$ are given by \begin{align}
     A^n_1 & = \sum_{e\in\Gamma_h} \frac{\tilde{\sigma}}{h_e}\|[\phi_h^n]\|^2_{L^2(e)} - \sum_{e\in\Gamma_h} \frac{\tilde{\sigma}}{h_e}\|[\phi_h^{n-1}]\|_{L^2(e)}^2, \label{eq:def_A1}\\ 
     A^n_2 & = \|\bm{G}_h([\phi_h^n]) \|^2 - \| \bm{G}_h([\phi_h^{n-1}])\|^2 \label{eq:def_A2}.
 \end{align}   
\end{lemma}
\begin{proof}
Using \eqref{eq:equivalent_def_Pi_h} in \eqref{eq:preposition}, we obtain \eqref{eq:div_error} for $n\geq 1$. 
 Substituting \eqref{eq:error_pressure_correction} in \eqref{eq:div_error}, we obtain \eqref{eq:div_error_2} for $n\geq 1$. To show \eqref{eq:error_diff}, we use \eqref{eq:sec_err_eq}: 
\begin{equation}
\begin{aligned}
&\|\err - \bm{e}_h^{n-1}\|^2 = (\err - \bm{e}_h^{n-1}, \err - \bm{e}_h^{n-1}) \label{eq:expanding_eh} \\
& = (\bm{e}_h^n - \bm{e}_h^{n-1} + \tau(\nabla_h \phi_h^n -\bm{G}_h([\phi_h^n])),\bm{e}_h^n - \bm{e}_h^{n-1} + \tau(\nabla_h\phi_h^n -\bm{G}_h([\phi_h^n])))   \\ 
& = \|\bm{e}_h^n - \bm{e}_h^{n-1} \|^2 + 2\tau(\bm{e}_h^n - \bm{e}_h^{n-1}, \nabla_h \phi_h^n -\bm{G}_h([\phi_h^n]))  + \tau^2 \|\nabla_h \phi_h^n -\bm{G}_h([\phi_h^n]) \|^2. 
\end{aligned}
\end{equation}
For the second term, we use the equivalent definition \eqref{eq:def_b_lift_2}.
\begin{align}
    (\bm{e}_h^n - \bm{e}_h^{n-1}, \nabla_h \phi_h^n -\bm{G}_h([\phi_h^n]))  = - b(\bm{e}_h^n - \bm{e}_h^{n-1}, \phi_h^n).
\end{align}
Since $\phi_h^0 = 0$, expanding the last term in \eqref{eq:expanding_eh} and using \eqref{eq:div_error_2} yield \eqref{eq:error_diff} for $n =1$. For $n \geq 2$, we  subtract \eqref{eq:div_error_2} at time step  $n$ from \eqref{eq:div_error_2} at time step $n-1$.  We obtain for $n \geq 2$ 
\begin{multline*}
-b(\bm{e}_h^n - \bm{e}_h^{n-1}, \phi_h^n) = \tau \sum_{e\in \Gamma_h} \frac{\tilde{\sigma}}{h_e} \int_e ([\phi_h^n]- [\phi_h^{n-1}]) [\phi_h^n] \\ - \tau (\bm{G}_h([\phi_h^n]) - \bm{G}_h([\phi_h^{n-1}]),\bm{G}_h([\phi_h^n])).  
\end{multline*} 
Thus, for $n\geq 2$,  we have 
\begin{multline}
    (\bm{e}_h^n - \bm{e}_h^{n-1}, \nabla_h \phi_h^n -\bm{G}_h([\phi_h^n]))  = \frac{\tau}{2} \sum_{e\in \Gamma_h} \frac{\tilde{\sigma}}{h_e} (\|[\phi_h^n]\|_{L^2(e)}^2 - \|[\phi_h^{n-1}]\|_{L^2(e)}^2  \\ + \|[ \phi_h^n - \phi_h^{n-1}]\|_{L^2(e)}^2) 
    - \frac{\tau}{2} ( \|\bm{G}_h([\phi_h^n]) \|^2 - \| \bm{G}_h([\phi_h^{n-1}])\|^2 + \| \bm{G}_h([\phi_h^n - \phi_h^{n-1}])\|^2 ). \label{eq:expanding_gh}
\end{multline}
Substituting \eqref{eq:expanding_gh} in \eqref{eq:expanding_eh} and expanding the last term in \eqref{eq:expanding_eh} yield \eqref{eq:error_diff} for $n \geq 2$.
\end{proof}
We proceed by presenting some bounds for the forms $a_{\mathcal{C}}, \mathcal{U}$, and $\mathcal{C}$. The proof of these bounds is inspired from the proof of proposition 4.1 in \cite{girault2005splitting}. However, the proof differs since in our scheme, we do not have that $b(\bm{e}_h^{n}, q_h) = 0$ for all $q_h \in M_h$.
\begin{lemma} \label{prep:bounds_nonlinear_terms}
Fix $\phi_h\in M_h$.  Assume that there is $\bm{w}_h \in \bm{X}_h$ such that
\begin{equation}
    b(\bm{w}_h, q_h ) =  -\tau \sum_{e \in \Gamma_h} \frac{\tilde{\sigma}}{h_e} \int_e [\phi_h ][q_h] + \tau(\bm{G}_h([\phi_h ]), \bm{G}_h([q_h])), \quad \forall q_h \in M_h.  \label{eq:assum_on_w}
\end{equation}
Then, there exists a constant $C$, independent of $h$, $\tau$, $\bm{w}_h, \bm{z}, \bm{v}$ and $\bm{\theta}_h$, such that the following estimates hold. 
\begin{enumerate}[(i)]
\item If $\bm{v} \in (W^{1,3}(\Omega)\cap H_0^1(\Omega) \cap L^{\infty}(\Omega))^d$, then for any $\bm{z} \in \bm{X}$ and $\bm{\theta}_h \in \bm{X}_h$:
    \begin{equation}
    |a_\mathcal{C}(\bm{z}; \bm{w}_h, \bm{v}, \bm{\theta}_h)| \leq 
C \left(\|\bm{w}_h \| + \tau \left( \sum_{e \in \Gamma_h} \frac{\tilde{\sigma}}{h_e} \| [ \phi_h ] \|_{L^2(e)}^2 \right)^{1/2}\right) 
\vertiii{\bm{v}} \| \bm{\theta}_h\|_{\DG}.
\label{eq:first_estimate_nonlinear_form}
\end{equation}
\item If $\bm{v} \in (H^{k+1}(\Omega) \cap  H^1_0(\Omega))^d$, then for any $\bm{\theta}_h \in \bm{X}_h$:
\begin{multline}
|\mathcal{C}(\bm{w}_h, \bm{v} - \Pi_h \bm{v}, \bm{\theta}_h)|\\ \leq C \left(\|\bm{w}_h\| + \tau \left( \sum_{e \in \Gamma_h} \frac{\tilde{\sigma}}{h_e} \| [\phi_h  ] \|_{L^2(e)}^2 \right)^{1/2} \right)\vertiii{\bm{v}}\| \bm{\theta}_h\|_{\DG}.\label{eq:second_estimate_nonlinear_term_C1}
\end{multline}
\end{enumerate}
\end{lemma}
\begin{proof}
(i) Since $\bm{v}$ belongs to $H^1(\Omega)^d$ and vanishes on the boundary, we have: 
\begin{equation}
 a_\mathcal{C}(\bm{z}; \bm{w}_h, \bm{v}, \bm{\theta}_h) =  \sum_{E \in \mesh_h} \int_{E} (\bm{w}_h \cdot \nabla \bm{v}) \cdot \bm{\theta}_h  + \frac{1}{2} b(\bm{w}_h, \bm{v} \cdot \bm{\theta}_h). \label{eq:ident_cont_u}
\end{equation}
With H\"{o}lder's inequality and \eqref{eq:discreter_poincare}, we have 
\begin{multline} \label{eq:first_nonlinear_estimate_0}
   \left|  \sum_{E \in \mesh_h} \int_{E} (\bm{w}_h \cdot \nabla \bm{v} ) \cdot \bm{\theta}_h \right| \leq  \| \bm{w}_h \|\| \nabla \bm{v} \|_{L^3(\Omega)} \|\bm{\theta}_h\|_{L^{6}(\Omega)}\\ \leq C_P\| \bm{w}_h \| \vert \bm{v} \vert_{W^{1,3}(\Omega)} \|\bm{\theta}_h\|_{\DG}. 
\end{multline}
To bound the second term in \eqref{eq:ident_cont_u}, we use a similar argument as in \cite{girault2005splitting}. Define $\bm{c}_1 , \bm{c}_2$ as piecewise constant vectors where on each element $ E \in \mesh_h$, 
\begin{equation}\label{eq:def_c1_c2}
    \bm{c}_1|_E = \frac{1}{|E|} \int_{E} \bm{v}, \quad \bm{c}_2|_E = \frac{1}{|E| } \int_E \bm{\theta}_h. 
\end{equation}
We then write, 
\begin{equation}
   \frac{1}{2} b(\bm{w}_h, \bm{v} \cdot \bm{\theta}_h) = \frac{1}{2} b(\bm{w}_h, \bm{v} \cdot \bm{\theta}_h - \bm{c}_1\cdot \bm{c}_2) + \frac{1}{2} b(\bm{w}_h, \bm{c}_1 \cdot \bm{c}_2). \label{eq:splitting_b_cont_u}
\end{equation}
The first term is bounded in Proposition 4.1 in \cite{girault2005splitting} in 2D domains. We follow a similar technique to obtain a bound in 2D/3D domains. We have, 
\begin{multline}
\frac{1}{2} b(\bm{w}_h, \bm{v}\cdot \bm{\theta}_h - \bm{c}_1 \cdot \bm{c}_2)  = \frac{1}{2}\sum_{E \in \mesh_h} \int_E (\nabla \cdot \bm{w}_h) (\bm{v}\cdot \bm{\theta}_h - \bm{c}_1 \cdot \bm{c}_2)\\ - \frac{1}{2}\sum_{e \in \Gamma_h \cup \partial \Omega} \int_e \{\bm{v}\cdot \bm{\theta}_h - \bm{c}_1 \cdot \bm{c}_2 \} [\bm{w}_h]\cdot\bm{n}_e. \label{eq:637}
\end{multline}
For the first term, note that by H\"{o}lder's, inverse and triangle inequalities, we have:
\begin{align}
 \Big|\int_{E} (\nabla \cdot &\bm{w}_h)   (\bm{v}\cdot \bm{\theta}_h - \bm{c}_1 \cdot \bm{c}_2)\Big| \nonumber\\
& \leq  C h_{E}^{-1}\|\bm{w}_h\|_{L^2(E)} (\|(\bm{v} - \bm{c}_1) \cdot \bm{\theta}_h\|_{L^2(E)} + \| \bm{c}_1\cdot(\bm{\theta}_h - \bm{c}_2)\|_{L^2(E)}) \nonumber  \\ 
& \leq  C h_{E}^{-1}\|\bm{w}_h\|_{L^2(E)}
(\|\bm{v}-\bm{c}_1\|_{L^{3}(E)} \| \bm{\theta}_h \|_{L^{6}(E)} + \| \bm{c}_1 \|_{L^{\infty}(E)} \| \bm{\theta}_h - \bm{c}_2\|_{L^2(E)}). \nonumber
\end{align}   
To proceed, we recall Poincare's inequality, see Theorems 1 and 2 in Section 5.8 in \cite{evans2010partial}. For $r \in [1,\infty], E \in \mesh_h$ and $\bm{\theta}$ in $W^{1,r}(E)$,  
\begin{equation}
   \Vert \bm{\theta} - \frac{1}{|E|}\int_{E} \bm{\theta} \Vert_{L^r(E)} \leq C h_E \| \nabla \bm{\theta}\|_{L^r(E)}. \label{eq:prop_zero_avg}
\end{equation}
Applying \eqref{eq:prop_zero_avg}, we find 
\begin{align}
    \left|\int_{E} (\nabla \cdot \bm{w}_h)(\bm{v}\cdot \bm{\theta}_h - \bm{c}_1 \cdot \bm{c}_2)\right| \leq C \|\bm{w}_h\|_{L^2(E)} (\vert \bm{v}\vert_{W^{1,3}(E)} \| \bm{\theta}_h\|_{L^{6}(E)} \\ + \|\bm{v}\|_{L^{\infty}(E)}\|\nabla \bm{\theta}_h\|_{L^2(E)}). \nonumber
\end{align} 
To handle the second term in the right-hand side of \eqref{eq:637}, consider a face $e\in \Gamma_h $ and let $E_e^1$ and $E_e^2$ denote the elements sharing $e$. Terms involving faces $e \in \partial \Omega$ are handled similarly.  With Cauchy-Schwarz's inequality, we have 
\begin{multline*}
    \vert \int_e \{ \bm{v}\cdot\bm{\theta}_h -\bm{c}_1\cdot\bm{c}_2\} [\bm{w}_h]\cdot\bm{n}_e \vert 
    \leq \frac12 \sum_{i,j = 1}^2 \Vert (\bm{v}\cdot\bm{\theta}_h -\bm{c}_1\cdot\bm{c}_2)|_{E_e^i} \Vert_{L^2(e)} \Vert \bm{w}_h|_{E_e^j} \Vert_{L^2(e)}.
\end{multline*} 
Let us consider the case $i = 1$ and $j=2$; the other terms are handled similarly:
\begin{align*}
\Vert (\bm{v}\cdot\bm{\theta}_h -\bm{c}_1\cdot\bm{c}_2)|_{E_e^1} \Vert_{L^2(e)} \Vert \bm{w}_h|_{E_e^2} \Vert_{L^2(e)}
\leq \Vert (\bm{v}-\bm{c}_1)|_{E_e^1} \cdot\bm{\theta}_h|_{E_e^1} \Vert_{L^2(e)} \Vert \bm{w}_h|_{E_e^2} \Vert_{L^2(e)}
\\
+ \Vert \bm{c}_1|_{E_e^1} \cdot(\bm{\theta}_h-\bm{c}_2)|_{E_e^1} \Vert_{L^2(e)} \Vert \bm{w}_h|_{E_e^2} \Vert_{L^2(e)}.
\end{align*}
With a trace inequality and H\"{o}lder's inequality, we have
\begin{align*}
\Vert (\bm{v}\cdot\bm{\theta}_h -\bm{c}_1\cdot\bm{c}_2)|_{E_e^1} \Vert_{L^2(e)} \Vert \bm{w}_h|_{E_e^2} \Vert_{L^2(e)}
\leq C \vert e \vert \vert E_e^1\vert^{-\frac12} \vert E_e^2\vert^{-\frac12} \Vert \bm{w}_h\Vert_{L^2(E_e^2)}\\
\times \left( (\Vert \bm{v}-\bm{c}_1\Vert_{L^3(E_e^1)} +  h_{E_e^1} \Vert \nabla \bm{v}\Vert_{L^3(E_e^1)}) \Vert \bm{\theta}_h\Vert_{L^6(E_e^1)} 
+  \Vert \bm{v}\Vert_{L^\infty(E_e^1)} \Vert \bm{\theta}_h - \bm{c}_2 \Vert_{L^2(E_e^1)}\right).
\end{align*} 
With \eqref{eq:prop_zero_avg}, this becomes
\begin{align}
\Vert (\bm{v}\cdot\bm{\theta}_h -\bm{c}_1\cdot\bm{c}_2)|_{E_e^1} \Vert_{L^2(e)} \Vert \bm{w}_h|_{E_e^2} \Vert_{L^2(e)}
\leq C \vert e \vert \vert E_e^1\vert^{-\frac12} \vert E_e^2\vert^{-\frac12} \Vert \bm{w}_h\Vert_{L^2(E_e^2)}\nonumber\\
\times h_{E_e^1}  \left(\Vert \nabla \bm{v}\Vert_{L^3(E_e^1)} \Vert \bm{\theta}_h\Vert_{L^6(E_e^1)} 
+ \Vert \bm{v}\Vert_{L^\infty(E_e^1)} \Vert \nabla \bm{\theta}_h\Vert_{L^2(E_e^1)}\right)
\nonumber\\
\leq C \Vert \bm{w}_h\Vert_{L^2(E_e^2)}  \left(\Vert \nabla \bm{v}\Vert_{L^3(E_e^1)} \Vert \bm{\theta}_h\Vert_{L^6(E_e^1)}
+ \Vert \bm{v}\Vert_{L^\infty(E_e^1)} \Vert \nabla \bm{\theta}_h\Vert_{L^2(E_e^1)}\right).
\label{eq:detailtrace}
\end{align} 
Therefore, with discrete H\"{o}lder's inequality and \eqref{eq:discreter_poincare}, we obtain 
\[
 \frac12 \sum_{e\in\Gamma_h\cup\partial\Omega} 
\Vert (\bm{v}\cdot\bm{\theta}_h -\bm{c}_1\cdot\bm{c}_2)|_{E_e^1} \Vert_{L^2(e)} \Vert \bm{w}_h|_{E_e^2} \Vert_{L^2(e)}
\leq C \Vert \bm{w}_h\Vert \, \vertiii{\bm{v}} \|\bm{\theta}_h \|_{\DG}.
\]
We bound the other terms similarly, and combining the bounds, we obtain
\begin{align}
    \frac{1}{2} |b(\bm{w}_h, \bm{v} \cdot \bm{\theta}_h - \bm{c}_1\cdot \bm{c}_2)| \leq C  \|\bm{w}_h\|\vertiii{\bm{v}} \|\bm{\theta}_h \|_{\DG}.\label{eq:first_nonlinear_estimate_1}
\end{align}
It remains to handle the second term in the right-hand side of \eqref{eq:splitting_b_cont_u}. To this end, we use \eqref{eq:assum_on_w}. 
\begin{equation}
\frac{1}{2} b(\bm{w}_h, \bm{c}_1 \cdot \bm{c}_2) = -\frac{\tau}{2} \sum_{e \in \Gamma_h} \frac{\tilde{\sigma}}{h_e} \int_e [ \phi_h ][\bm{c}_1 \cdot \bm{c}_2] + \frac{\tau}{2}(\bm{G}_h([ \phi_h ]), \bm{G}_h([\bm{c}_1 \cdot \bm{c}_2])).
\end{equation}
Using Cauchy-Schwarz's inequality and \eqref{eq:lift_prop_g}, we have 
\begin{equation} \label{eq:first_bd_b_c1c2}
    \frac{1}{2} |b(\bm{w}_h, \bm{c}_1 \cdot \bm{c}_2)|  \leq C \tau \left( \sum_{e \in \Gamma_h} \frac{\tilde{\sigma}}{h_e} \| [ \phi_h ] \|_{L^2(e)}^2 \right)^{1/2} \left( \sum_{e \in \Gamma_h} \frac{\tilde{\sigma}}{h_e} \| [\bm{c}_1 \cdot \bm{c}_2] \|_{L^2(e)}^2 \right)^{1/2}. 
\end{equation}
We write
\begin{align}
 \sum_{e \in \Gamma_h} \frac{\tilde{\sigma}}{h_e} \| [\bm{c}_1 \cdot \bm{c}_2] \|_{L^2(e)}^2  & = \sum_{e\in \Gamma_h} \frac{\tilde{\sigma}}{h_e}\| [\bm{c}_1 \cdot (\bm{c}_2 - \bm{\theta}_h)] + [\bm{c}_1 \cdot \bm{\theta}_h]\|^2_{L^2(e)} \nonumber \\ 
 & \leq 2 \sum_{e\in \Gamma_h} \frac{\tilde{\sigma}}{h_e}  \| [\bm{c}_1 \cdot (\bm{c}_2 - \bm{\theta}_h)]  \|^2_{L^2(e)} + 2 \sum_{e\in\Gamma_h} \frac{\tilde{\sigma}}{h_e} \|[\bm{c}_1 \cdot \bm{\theta}_h]\|^2_{L^2(e)} \nonumber\\  
 &= T_1 + T_2. \label{eq:def_T1_T2}
\end{align}
To bound $T_1$, we apply a trace inequality and \eqref{eq:prop_zero_avg}, as it was done in \eqref{eq:detailtrace}. 
Consider the elements $E_e^1$ and $E_e^2$ that share  $e$. We have, for $i=1,2$:
\begin{align*}
    \frac{1}{\sqrt{h_e}} \|\bm{c}_1 \cdot (\bm{c}_2 - \bm{\theta}_h) \vert_{E_e^i} \|_{L^2(e)} 
&\leq C \vert e \vert^{\frac12 - \frac{1}{2(d-1)}} \vert E_e^i\vert^{-\frac12} h_{E_e^i} \|\bm{v}\|_{L^{\infty}(E_e^i)} \| \nabla \bm{\theta}_h \|_{L^2(E_e^i)}\\
&\leq C  \|\bm{v}\|_{L^{\infty}(E_e^i)} \| \nabla \bm{\theta}_h \|_{L^2(E_e^i)}.
\end{align*}
Hence, we obtain the following bound on $T_1$. 
\begin{align}  \label{eq:bd_T1}
    T_1 &\leq C \| \bm{v} \|^2_{L^{\infty}(\Omega)} \|\bm{\theta}_h\|^2_{\DG}.
\end{align} 
To handle $T_2$, we split it as such. 
\begin{align}
    [\bm{c}_1 \cdot \bm{\theta}_h] & = [(\bm{c}_1 - \bm{v})\cdot \bm{\theta}_h] + [\bm{v} \cdot \bm{\theta}_h]. \nonumber 
\end{align}
Thus, we have 
\begin{align}
    T_2 &\leq 4 \sum_{e\in \Gamma_h} \frac{\tilde{\sigma}}{h_e} \| [(\bm{c}_1 - \bm{v})\cdot \bm{\theta}_h] \|^2_{L^2(e)}  + 4\sum_{e\in \Gamma_h} \frac{\tilde{\sigma}}{h_e} \|[\bm{v} \cdot \bm{\theta}_h]\|^2_{L^2(e)} =  T_2^1 + T_2^2. \label{eq:def_t21_t22}
\end{align}
To bound $T_2^1$, we will use a trace inequality, H\"{o}lder's inequalities, \eqref{eq:prop_zero_avg}, and \eqref{eq:discreter_poincare}. 
We skip the details as the argument is similar to the one used in \eqref{eq:detailtrace}.
\begin{equation}\label{eq:bd_t21}
    T_2^1 \leq C (\vert \bm{v}\vert^2_{W^{1,3}(\Omega)} + \|\bm{v}\|^2_{L^\infty(\Omega)}) \| \bm{\theta}_h \|^2_{\DG}.
\end{equation}
The term $T_2^2$ is simply bounded as such, 
\begin{equation}\label{eq:bd_t22}
T_2^2 \leq 4 \| \bm{v}\|^2_{L^{\infty}(\Omega)} \sum_{e\in\Gamma_h} \frac{\tilde{\sigma}}{h_e} \|[\bm{\theta}_h]\|^2_{L^2(e)} \leq 4 \| \bm{v}\|^2_{L^{\infty}(\Omega)}  \| \bm{\theta}_h \|_{\DG}^2.  
\end{equation}
Substituting bounds \eqref{eq:bd_T1}, \eqref{eq:bd_t21} and \eqref{eq:bd_t22} in \eqref{eq:first_bd_b_c1c2}, we obtain 
\begin{equation}\label{eq:first_nonlinear_estimate_2}
\frac{1}{2}| b(\bm{w}_h, \bm{c}_1 \cdot \bm{c}_2) | \leq C\tau \vertiii{\bm{v}}\|
\bm{\theta}_h\|_{\DG} \left(\sum_{e\in \Gamma_h} \frac{\tilde{\sigma}}{h_e} \|[\phi_h] \|^2_{L^2(e)}\right)^{1/2}.
\end{equation}
With \eqref{eq:first_nonlinear_estimate_1} and \eqref{eq:first_nonlinear_estimate_2}, we have the following bound: 
\begin{equation}
\frac{1}{2} |b(\bm{w}_h, \bm{v} \cdot \bm{\theta}_h)| \leq C  \left(\| \bm{w}_h \| + \left(\sum_{e\in \Gamma_h} \frac{\tilde{\sigma}}{h_e} \|[\phi_h] \|^2_{L^2(e)}\right)^{1/2}\right)\vertiii{\bm{v}}\|\bm{\theta}_h\|_{\DG}.  \label{eq:bound_form_b}
\end{equation}
We conclude the proof of \eqref{eq:first_estimate_nonlinear_form} by combining bounds \eqref{eq:first_nonlinear_estimate_0} and \eqref{eq:bound_form_b}. 

(ii) First, we remark that by a Sobolev embedding, $\bm{v} \in (W^{1,3}(\Omega) \cap L^{\infty}(\Omega))^d$ since $k\geq 1$. We have, \begin{multline}
 \mathcal{C}(\bm{w}_h, \bm{v}- \Pi_h \bm{v}, \bm{\theta}_h) =  \sum_{E \in \mesh_h} \int_{E} (\bm{w}_h \cdot \nabla (\bm{v} - \Pi_h\bm{v} ) )\cdot \bm{\theta}_h \\   + \frac{1}{2} b(\bm{w}_h, (\bm{v} -\Pi_h \bm{v}) \cdot \bm{\theta}_h). \label{eq:expanding_C} 
\end{multline}
The first term is bounded exactly like in the proof of (i) (see the derivation of bound \eqref{eq:first_nonlinear_estimate_0}) with $\bm{v}$ replaced by $\bm{v}- \Pi_h\bm{v}$. By the stability of the interpolant, we have
\begin{equation}
     \sum_{E \in \mesh_h} \left| \int_{E} (\bm{w}_h \cdot \nabla (\bm{v}- \Pi_h\bm{v} ) )\cdot \bm{\theta}_h \right|  
 \leq C \|\bm{w}_h\| |\bm{v}|_{W^{1,3}(\Omega)} \| \bm{\theta}_h \|_{\DG}. \label{eq:third_nonlinear_term_0}
\end{equation}
The second term is split in the same way as in the proof of (i). 
\begin{multline}
    \frac{1}{2} b(\bm{w}_h, (\bm{v} -\Pi_h \bm{v}) \cdot \bm{\theta}_h)   =    \frac{1}{2} b(\bm{w}_h, (\bm{v} -\Pi_h \bm{v}) \cdot \bm{\theta}_h - \tilde{\bm{c}}_1 \cdot \bm{c}_2) \\ + \frac{1}{2} b(\bm{w}_h, \tilde{\bm{c}}_1 \cdot \bm{c}_2), 
\end{multline}
where $\bm{c}_2$ is defined in \eqref{eq:def_c1_c2} and $\tilde{\bm{c}}_1 $ is a piecewise constant vector defined on each element $E \in \mesh_h$ as $$\tilde{\bm{c}}_1 \vert_{E}  = \frac{1}{|E|} \int_E (\bm{v} - \Pi_h \bm{v}).$$
Similar to \eqref{eq:first_nonlinear_estimate_1} and using the stability of the interpolant \eqref{eq:stability_linf_pih}, we have the bound 
\begin{multline}
\frac{1}{2} |b(\bm{w}_h, (\bm{v} -\Pi_h \bm{v}) \cdot \bm{\theta}_h - \tilde{\bm{c}}_1 \cdot \bm{c}_2) | \\ \leq C\|\bm{w}_h \|(\| \nabla_h(\bm{v} - \Pi_h \bm{v})\|_{L^3(\Omega)}+ \|\bm{v} - \Pi_h \bm{v} \|_{L^\infty(\Omega)}) \|\bm{\theta}_h\|_{\DG} 
 \leq C \|\bm{w}_h \| \, \vertiii{\bm{v}} \, \|\bm{\theta}_h \|_{\DG}.\label{eq:second_nonlinear_estimate_b}  
\end{multline}
With similar arguments as equations \eqref{eq:first_bd_b_c1c2}, \eqref{eq:def_T1_T2} and \eqref{eq:def_t21_t22},  we have 
\begin{equation}
\frac{1}{2} | b(\bm{w}_h, \tilde{\bm{c}}_1 \cdot \bm{c}_2)| \leq C \tau  \left( \sum_{e \in \Gamma_h} \frac{\tilde{\sigma}}{h_e} \| [ \phi_h] \|_{L^2(e)}^2 \right)^{1/2}(\tilde{T}_1 + \tilde{T}_2^1 +\tilde{T}_2^2)^{1/2}.
\end{equation}
Here $\tilde{T}_1$,  $\tilde{T}_2^1$ and $\tilde{T}_2^2$ are defined in a similar way to $T_1,T_2^1$ and $T_2^2$ with $\bm{c}_1$ replaced by $\tilde{\bm{c}}_1$ and $\bm{v}$ replaced by $\bm{v} - \Pi_h\bm{v}$. The terms $\tilde{T}_1$ and $\tilde{T}_2^1$ are bounded exactly like $T_1$ \eqref{eq:bd_T1} and $T_2^1$ \eqref{eq:bd_t21} with $\bm{v}$ replaced by $\bm{v} - \Pi_h\bm{v}$. With the stability of the interpolant, we have 
\begin{align*}
|\tilde{T}_1| &\leq C\|\bm{v} - \Pi_h \bm{v}\|^2_{L^{\infty}(\Omega)}\|\bm{\theta}_h\|^2_{\DG}\leq C\vertiii{\bm{v}}^2\|\bm{\theta}_h\|^2_{\DG},
\\ 
|\tilde{T}^1_2|& \leq C \left(  \| \nabla_h (\bm{v} - \Pi_h \bm{v}) \|_{L^3(\Omega)}^2 
+\| \bm{v}-\Pi_h \bm{v}\|_{L^{\infty}(\Omega)}^2 \right)\|\bm{\theta}_h \|_{\DG}^2 \leq  
  C \vertiii{\bm{v}}^2\|\bm{\theta}_h \|_{\DG}^2.   
\end{align*}
The only term that differs is $\tilde{T}_2^2$. We have, 
\begin{align}
    \tilde{T}_2^2 &= 4 \sum_{e\in\Gamma_h} \frac{\tilde{\sigma}}{h_e} \| [(\bm{v} - \Pi_h \bm{v}) \cdot \bm{\theta}_h ] \|^2_{L^2(e)}.  
\end{align}
Consider an edge $e$ and an adjacent element $E_e^1$. Using trace and H\"{o}lder's inequalities, we have 
\begin{align*}
    \| (&\bm{v} - \Pi_h \bm{v}) \cdot \bm{\theta}_h \vert_{E_e^1} \|_{L^2(e)}  \leq C h_{E_e^1}^{-1/2} \|\bm{v} - \Pi_h \bm{v}\|_{L^3(E_e^1)} \| \bm{\theta}_h \|_{L^{6}(E_e^1)}\\ &   + Ch_{E_e^1}^{1/2} \| \nabla(\bm{v} - \Pi_h \bm{v})\|_{L^{3}(E_e^1)} \| \bm{\theta}_h \|_{L^{6}(E_e^1)}  + Ch_{E_e^1}^{1/2} \| \bm{v} - \Pi_h \bm{v} \|_{L^\infty(E_e^1)} \|\nabla \bm{\theta}_h \|_{L^2(E_e^1)}, \nonumber  
\end{align*}
With the approximation property \eqref{eq:approximation_prop_local_1}, we have: 
\begin{equation}
\|\bm{v} - \Pi_h \bm{v}\|_{L^3(E_e^1)} \leq C h_{E_e^1} \vert \bm{v}\vert_{W^{1,3}(\Delta_{E_e^1})}. 
\end{equation}
Hence, with discrete H\"{o}lder's inequality, \eqref{eq:discreter_poincare}, \eqref{eq:approximation_prop_local_2}, and  \eqref{eq:stability_linf_pih},  $\tilde{T}_2^2$ is bounded as follows. 
\begin{equation}
|\tilde{T}_2^2| \leq C (|\bm{v}|^2_{W^{1,3}(\Omega)} + \| \bm{v}\|^2_{L^{\infty}(\Omega)})\|\bm{\theta}_h \|^2_{\DG}.
\end{equation}
Then, we have 
\begin{equation}\label{eq:third_nonlinear_estimate_1}
    \frac{1}{2} | b(\bm{w}_h, \tilde{\bm{c}}_1 \cdot \bm{c}_2)| \leq C \tau \left( \sum_{e \in \Gamma_h} \frac{\tilde{\sigma}}{h_e} \| [ \phi_h] \|_{L^2(e)}^2 \right)^{1/2} \vertiii{\bm
    {v}}\|\bm{\theta}_h\|_{\DG}. 
\end{equation}
Combining bounds \eqref{eq:third_nonlinear_term_0}, \eqref{eq:second_nonlinear_estimate_b}, \eqref{eq:third_nonlinear_estimate_1}, we obtain the bound \eqref{eq:second_estimate_nonlinear_term_C1}. 
\end{proof}
In addition, we have the following bounds on $a_\mathcal{C}$ and $\mathcal{C}$ presented in the next lemma. 
\begin{lemma}\label{lemma:bounds_nonlinear_terms_2}
  There exists a constant $C$, independent of $h, \tau, \bm{w}, \bm{z}, \bm{v}, \bm{w}_h,$ and $\bm{\theta}_h$ such that the following bounds hold. 
 \begin{enumerate}[(i)]
      \item If $\bm{w} \in (H^{k+1}(\Omega)\cap H^1_0(\Omega))^d$  satisfies $b(\bm{w}, q_h) = 0,\,\, \forall q_h \in M_h$ and if $\bm{v} \in (W^{1,3}(\Omega)\cap H^1_0(\Omega)\cap L^{\infty}(\Omega))^d$, then for any $\bm{z} \in \bm{X}$ and $ \bm{\theta}_h \in \bm{X}_h$:
     \begin{equation}
      |a_\mathcal{C}(\bm{z}; \Pi_h\bm{w}- \bm{w}, \bm{v}, \bm{\theta}_h )| \leq C  h^{k+1} \vert \bm{w} \vert_{H^{k+1}(\Omega)} \vertiii{\bm{v}} \|\bm{\theta}_h\|_{\DG}. \label{eq:third_estimate_nonlinear_term}
     \end{equation}
     \item If $\bm{w} \in (H^{k+1}(\Omega) \cap  H^1_0(\Omega))^d$ and $\bm{v} \in (W^{1,3}(\Omega) \cap H^1_0(\Omega) \cap L^{\infty}(\Omega))^d$, then for any $\bm{z} \in \bm{X}$ and $\bm{\theta}_h \in \bm{X}_h$:
     \begin{align}
     |\mathcal{C}(\Pi_h \bm{v},\bm{w} - \Pi_h \bm{w},  \bm{\theta}_h) | &\leq C h^{k}|\bm{w}|_{H^{k+1}(\Omega)}\vertiii{\bm{v}}\| \bm{\theta}_h\|_{\DG}, \label{eq:second_estimate_nonlinear_term_C2}  \\ 
     |\mathcal{U}(\bm{z};\Pi_h \bm{v},\bm{w} - \Pi_h \bm{w},  \bm{\theta}_h) | &\leq C h^{k}|\bm{w}|_{H^{k+1}(\Omega)} \vertiii{\bm{v}}\| \bm{\theta}_h\|.\label{eq:second_estimate_nonlinear_term_U2}
     \end{align} 
     \item If $\bm{v} \in H^{k+1}(\Omega)^d$, then for any $\bm{z} \in \bm{X}$ and $\bm{w}_h, \bm{\theta}_h \in \bm{X}_h$:  
     \begin{align}
      |\mathcal{U}(\bm{z}; \bm{w}_h, \bm{v}- \Pi_h \bm{v}, \bm{\theta}_h)|   \leq C h^{k-d/4}\| \bm{w}_h \||\bm{v}|_{H^{k+1}(\Omega)} \|\bm{\theta}_h \|_{\DG}. \label{eq:second_estimate_nonlinear_term} 
     \end{align}
 \end{enumerate}
\end{lemma}
\begin{proof}
(i) Note that since $\bm{v}$ has zero jumps and vanishes on the boundary, the upwind term vanishes. We have  
\begin{equation}
    a_\mathcal{C}(\bm{z}; \Pi_h \bm{w} - \bm{w}, \bm{v}, \bm{\theta}_h) =  \sum_{E \in \mesh_h} \int_{E} ((\Pi_h \bm{w} - \bm{w} )\cdot \nabla \bm{v} ) \cdot \bm{\theta}_h  + \frac{1}{2} b(\Pi_h \bm{w} - \bm{w} , \bm{v} \cdot \bm{\theta}_h). \nonumber
\end{equation}
The first term is bounded similarly as before (with \eqref{eq:discreter_poincare} and \eqref{eq:approximation_prop_1}:  
\begin{align}
\left| \sum_{E \in \mesh_h} \int_{E} ((\Pi_h \bm{w} - \bm{w} )\cdot \nabla \bm{v} ) \cdot \bm{\theta}_h \right| &\leq  \|\Pi_h \bm{w} - \bm{w}\| \vert \bm{v}\vert_{W^{1,3}(\Omega)} \| \bm{\theta}_h \|_{L^6(\Omega)} \nonumber  \\ 
& \leq  Ch^{k+1}|\bm{w} |_{H^{k+1}(\Omega)} \vert \bm{v} \vert_{W^{1,3}(\Omega)} \| \bm{\theta}_h \|_{\DG} \label{eq:nonlinear_term_3_0}.
\end{align}
Recall the definition of $\bm{c}_1, \bm{c}_2$ in \eqref{eq:def_c1_c2}. Since $b(\Pi_h \bm{w} - \bm{w} , q_h ) = 0$ for any $q_h \in M_h$, we have 
\begin{equation}
    b(\Pi_h \bm{w} - \bm{w} , \bm{v} \cdot \bm{\theta}_h) =  b(\Pi_h \bm{w} - \bm{w}, \bm{v} \cdot \bm{\theta}_h  - \bm{c}_1 \cdot \bm{c}_2), \label{eq:adding_zero}
\end{equation}
Using H\"{o}lder's inequality and \eqref{eq:prop_zero_avg}, we have the following bounds. For $E \in \mesh_h$, 
\begin{align*}
\|\bm{v} \cdot \bm{\theta}_h - \bm{c}_1 \cdot \bm{c}_2 \|_{L^2(E)} &\leq \|(\bm{v} - \bm{c}_1) \cdot \bm{\theta}_h \|_{L^2(E)} + \|\bm{c}_1 \cdot (\bm{c}_2 - \bm{\theta}_h)\|_{L^2(E)} \\
& \leq \|\bm{v} - \bm{c}_1 \|_{L^3(E)} \| \bm{\theta}_h \|_{L^{6}(E)} + \|\bm{c}_1\|_{L^{\infty}(E)} \|\bm{c}_2 - \bm{\theta}_h \|_{L^2(E)} \\
& \leq Ch_E |\bm{v}|_{W^{1,3}(E)} \|\bm{\theta}_h \|_{L^{6}(E)}
+ Ch_E \|\bm{v} \|_{L^{\infty}(E)} \| \nabla \bm{\theta}_h\|_{L^2(E)}. \end{align*}
Hence, with the help of the discrete H\"{o}lder's inequality and \eqref{eq:discreter_poincare}, the volume term in the right-hand side of \eqref{eq:adding_zero} is bounded as follows. 
\begin{align}
 \Big| \sum_{E \in \mesh_h} \int_{E} & (\nabla \cdot (\Pi_h \bm{w}- \bm{w}))(\bm{v} \cdot \bm{\theta}_h - \bm{c}_1 \cdot \bm{c}_2) \Big|  \nonumber\\ 
&\leq C \sum_{E \in \mesh_h}\| \nabla (\Pi_h \bm{w} - \bm{w} )\|_{L^2(E)} \| \bm{v} \cdot \bm{\theta}_h - \bm{c}_1 \cdot \bm{c}_2\|_{L^2(E)}\nonumber\\ 
&\leq C h^{k+1} |\bm{w}|_{H^{k+1}(\Omega)}\vertiii{\bm{v}} \|\bm{\theta}_h \|_{\DG}.\label{eq:nonlinear_term_3_1}
\end{align} 
The face terms in the right-hand side of \eqref{eq:adding_zero} are handled in a slightly different way as in \eqref{eq:detailtrace}, with $\bm{w}_h$ replaced by $\Pi_h \bm{w} - \bm{w}$.  For instance, one of the terms is bounded as follows:
\begin{align*}
\Vert ((\bm{v}-\bm{c}_1) \cdot\bm{\theta}_h)|_{E_e^1}\Vert_{L^2(e)} 
\Vert (\Pi_h\bm{w}-\bm{w})|_{E_e^2}\Vert_{L^2(e)}
\leq C \vert e \vert \, \vert E_e^1\vert^{-\frac12} \vert E_e^2\vert^{-\frac12} h_{E_e^1} \\
\times (\Vert \nabla \bm{v}\Vert_{L^3(E_e^1)} \Vert \bm{\theta}_h\Vert_{L^6(E_e^1)} + \Vert \bm{v}\Vert_{L^\infty(E_e^1)} \Vert 
\nabla \bm{\theta}_h \Vert_{L^2(E_e^1)})
\\
\times (\Vert \Pi_h \bm{w}-\bm{w}\Vert_{L^2(E_e^2)} + h_{E_e^2} \Vert \nabla (\Pi_h \bm{w}-\bm{w})\Vert_{L^2(E_e^2)}).
\end{align*}
With \eqref{eq:approximation_prop_local_1}, \eqref{eq:approximation_prop_local_2}, we have
\begin{align*}
\Vert ((\bm{v}-\bm{c}_1) \cdot\bm{\theta}_h)|_{E_e^1}\Vert_{L^2(e)}
\Vert (\Pi_h\bm{w}-\bm{w})|_{E_e^2}\Vert_{L^2(e)}
\leq C h^{k+1} \Vert \bm{w}\Vert_{H^{k+1}(\Delta_{E_e^2})}\\
\times
(\Vert \nabla \bm{v}\Vert_{L^3(E_e^1)} \Vert \bm{\theta}_h\Vert_{L^6(E_e^1)} + \Vert \bm{v}\Vert_{L^\infty(E_e^1)} \Vert
\nabla \bm{\theta}_h \Vert_{L^2(E_e^1)}). 
\end{align*}
Hence, with discrete H\"{o}lder inequality and \eqref{eq:discreter_poincare}, the face  terms are bounded as follows. 
\begin{multline}
\frac{1}{2} \left| \sum_{e\in \Gamma_h \cup \partial \Omega} \int_e [\Pi_h \bm{w} - \bm{w}]\cdot \bm{n}_e \{ \bm{v} \cdot \bm{\theta}_h - \bm{c}_1 \cdot \bm{c}_2\} \right| 
 \\  \leq C h^{k+1}|\bm{w}|_{H^{k+1}(\Omega)}\vertiii{\bm{v}}\|\bm{\theta}_h \|_{\DG}.\label{eq:nonlinear_term_3_2}
\end{multline}
With bounds \eqref{eq:nonlinear_term_3_1} and \eqref{eq:nonlinear_term_3_2}, we obtain: 
\begin{equation}
\frac{1}{2}|b(\Pi_h \bm{w} - \bm{w}, \bm{v} \cdot \bm{\theta}_h)| \leq C h^{k+1}|\bm{w}|_{H^{k+1}(\Omega)}\vertiii{\bm{v}}\|\bm{\theta}_h \|_{\DG}. \label{eq:bound_b_pi_h}
\end{equation}
Combining bounds \eqref{eq:nonlinear_term_3_0} and \eqref{eq:bound_b_pi_h} yields the result.

(ii) To show \eqref{eq:second_estimate_nonlinear_term_C2}, we write 
\begin{multline}
\mathcal{C}(\Pi_h \bm{v}, \bm{w}- \Pi_h \bm{w}, \bm{\theta}_h) =  \sum_{E \in \mesh_h} \int_{E} (\Pi_h \bm{v} \cdot \nabla (\bm{w} - \Pi_h\bm{w} ) )\cdot \bm{\theta}_h  \\+ \frac{1}{2} b(\Pi_h \bm{v}, (\bm{w} -\Pi_h \bm{w}) \cdot \bm{\theta}_h). 
\end{multline}
With H\"{o}lder's inequality, stability of the interpolant and \eqref{eq:approximation_prop_1}, the first term is bounded by: 
\begin{equation}
    \| \Pi_h \bm{v}\|_{L^{\infty}(\Omega)}\|\nabla_h (\bm{w} - \Pi_h \bm{w})\| \|\bm{\theta}_h \| \leq C \vertiii{\bm{v}} h^{k} \vert \bm{w}\vert_{H^{k+1}(\Omega)} \|\bm{\theta}_h \|_{\DG}.\label{eq:bd_first_term_c2}
\end{equation}
To handle the second term, we will use \eqref{eq:equiv_form_b}. We have: 
\begin{multline}
 b(\Pi_h \bm{v}, (\bm{w} -\Pi_h \bm{w}) \cdot \bm{\theta}_h) = -\sum_{E \in \mesh_h} \int_E  \Pi_h \bm{v} \cdot \nabla((\bm{w} - \Pi_h \bm{w})\cdot \bm{\theta}_h ) \\ + \sum_{e \in \Gamma_h } \int_e \{ \Pi_h \bm{v} \} \cdot \bm{n}_e [(\bm{w} - \Pi_h \bm{w})\cdot \bm{\theta}_h ] = B_1 + B_2.\end{multline}
We bound $B_1$ with H\"{o}lder's inequality, stability of the interpolant \eqref{eq:stability_linf_pih}, and the approximation properties \eqref{eq:approximation_prop_1}-\eqref{eq:approximation_prop_2}. We have 
\begin{align} 
|B_1| & \leq \|\Pi_h \bm{v}\|_{L^{\infty}(\Omega)} (\| \nabla_h (\bm{w} - \Pi_h \bm{w})\| \| \bm{\theta}_h \| + \|\bm{w} - \Pi_h \bm{w} \|\| \nabla_h \bm{\theta}_h \| ) \nonumber, \\  
&  \leq Ch^{k}  \vertiii{\bm{v}}\vert \bm{w} \vert_{H^{k+1}(\Omega)} \|\bm{\theta}_h \|_{\DG}.\label{eq:bound_b1_c2} 
\end{align}
For $B_2$, 
let $e$ be an interior face in $\Gamma_h$ shared by $E_e^1$ and $E_e^2$. We have, 
\[
\Vert [(\bm{w} - \Pi_h \bm{w})\cdot \bm{\theta}_h ]\Vert_{L^1(e)}
\leq \Vert (\bm{w} - \Pi_h \bm{w})|_{E_e^1}\cdot \bm{\theta}_h|_{E_e^1}\Vert_{L^1(e)} +
\Vert (\bm{w} - \Pi_h \bm{w})|_{E_e^2}\cdot \bm{\theta}_h|_{E_e^2}\Vert_{L^1(e)}. 
\]
With H\"{o}lder's inequality and a trace estimate, we have for $i=1,2$:
\begin{align*}
\Vert& (\bm{w} - \Pi_h \bm{w})|_{E_e^i}\cdot \bm{\theta}_h|_{E_e^i}\Vert_{L^1(e)}  
\leq C h_{E_e^i}^{-1/2} \Vert  (\bm{w} - \Pi_h \bm{w})|_{E_e^i}\Vert_{L^2(e)} \, \Vert \bm{\theta}_h\Vert_{L^2(E_e^i)}\\
&\leq C (h_{E_e^i}^{-1} \Vert \bm{w} - \Pi_h \bm{w}\Vert_{L^2(E_e^i)}
+ \Vert \nabla (\bm{w} - \Pi_h \bm{w})\Vert_{L^2(E_e^i)})\|\bm{\theta}_h\|_{L^2(E_e^i)}. 
\end{align*}
We apply the approximation properties, sum over the faces, use \eqref{eq:discreter_poincare}, and obtain
\begin{equation}
|B_2|\leq Ch^{k} \| \Pi_h \bm{v} \|_{L^{\infty}(\Omega)} \vert \bm{w}\vert_{H^{k+1}(\Omega)} \|\bm{\theta}_h \|_{\DG} \leq C h^k \vertiii{\bm{v}}\vert \bm{w}\vert_{H^{k+1}(\Omega)} \|\bm{\theta}_h \|_{\DG} . \label{eq:bound_b2_c2}
\end{equation}
Bounds \eqref{eq:bd_first_term_c2}, \eqref{eq:bound_b1_c2}, and \eqref{eq:bound_b2_c2} yield \eqref{eq:second_estimate_nonlinear_term_C2}. To handle the upwind term, we consider the faces' contributions to the upwind terms. For details on such contributions, see Propositon 4.10 in \cite{girault2009dg}. We bound it by:
\begin{multline*}
C\|\Pi_h \bm{v}\|_{L^{\infty}(\Omega)} \sum_{i,j= 1}^{2} \sum_{e \in \Gamma_h} \|\bm{w} - \Pi_h \bm{w}|_{E_e^i}\|_{L^2(e)} \|\bm{\theta}_h |_{E_e^j}\|_{L^2(e)} \\ + C\|\Pi_h \bm{v}\|_{L^{\infty}(\Omega)}\sum_{ e\in\partial\Omega } \|\bm{w} - \Pi_h \bm{w}|_{E_e}\|_{L^2(e)} \|\bm{\theta}_h |_{E_e}\|_{L^2(e)}. 
\end{multline*} 
Consider the case $e \in \Gamma_h$. With a trace inequality and approximation properties \eqref{eq:approximation_prop_local_1}-\eqref{eq:approximation_prop_local_2}, we have
\begin{align*}
    \|\bm{w} - \Pi_h \bm{w}|_{E_e^i}\|_{L^2(e)} \|\bm{\theta}_h |_{E_e^j}\| \leq C |e||E_e^i|^{-1/2}|E_e^j|^{-1/2} h_{E_e^{i}}^{k+1} |\bm{w}|_{H^{k+1}(\Delta_{E_e^i})} \|\bm{\theta}_h\|_{L^2(E_e^j)} \\ 
    \leq Ch_{E_e^i}^{k}  |\bm{w}|_{H^{k+1}(\Delta_{E_e^i})} \|\bm{\theta}_h\|_{L^2(E_e^j)}.
\end{align*}
With similar bounds for the case $e \in \partial \Omega$, discrete H\"{o}lder's inequality, and \eqref{eq:stability_linf_pih},  we conclude that \eqref{eq:second_estimate_nonlinear_term_U2} holds. 

(iii) This bound can be found in Proposition 4.1 in \cite{girault2005splitting} for $k=1$ and $d=2$. Here, we extend it to any polynomial degree $k$ and for $d = 2,3$. We will use local trace and inverse estimates. For $i = 1,2$,  we have: 
\begin{alignat}{2}
    \| \bm{\theta}_h \vert_{E_e^i}\|_{L^4(e)} &\leq C |e|^{1/4}|E_e^i|^{-1/4} \|\bm{\theta}_h\|_{L^4(E_e^i)}, && \quad \forall \bm{\theta}_h \in \bm{X}_h ,  \label{eq:trace_estimate_l4} \\ 
    \| \bm{\theta}_h \|_{L^4(E_e^i)} &\leq C |E_e^i|^{-1/4} \|\bm{\theta}_h\|_{L^2(E_e^i)}, && \quad \forall \bm{\theta}_h \in \bm{X}_h  \label{eq:inverse_estimate_l4}.
\end{alignat}
We also consider the faces' contributions to handle the upwind term. With H\"{o}lder's inequality, we bound the upwind term, $ \mathcal{U}(\bm{z}; \bm{w}_h, \bm{v}- \Pi_h \bm{v}, \bm{\theta}_h)$  by: 
\begin{multline*} C \sum_{i,j,k = 1}^{2} \sum_{e \in \Gamma_h } \|\bm{w}_h|_{E_e^{i}}\|_{L^4(e)} \|(\bm{v} - 
    \Pi_h \bm{v})|_{E_e^{k}}\|_{L^2(e)}\|\bm{\theta}_h|_{E_e^{j}}\|_{L^4(e)} \\ + C \sum_{e \in \partial \Omega} \|\bm{w}_h|_{E_e}\|_{L^4(e)} \|(\bm{v} - 
    \Pi_h \bm{v})|_{E_e}\|_{L^2(e)}\|\bm{\theta}_h|_{E_e}\|_{L^4(e)} . \end{multline*}
Consider the case when $e \in \Gamma_h$ and $j = k $. The other cases are handled similarly. 
We use  \eqref{eq:trace_estimate_l4}, \eqref{eq:inverse_estimate_l4}, and \eqref{eq:approximation_prop_local_1}-\eqref{eq:approximation_prop_local_2}: 
\begin{align}  &\|\bm{w}_h|_{E_e^{i}}\|_{L^4(e)} \|(\bm{v} - \Pi_h \bm{v})|_{E_e^{j}}\|_{L^2(e)}\|\bm{\theta}_h|_{E_e^{j}}\|_{L^4(e)}  \nonumber \\
&\leq C|e| |E_e^i|^{-1/2}|E_e^j|^{-3/4} h_{E_e^j}^{k+1}\|\bm{w}_h\|_{L^2(E_e^i)} |\bm{v}|_{H^{k+1}(\Delta_{E_e^j})} \|\bm{\theta}_h\|_{L^4(E_e^j)} \nonumber . 
\end{align}
Since $|e| |E_e^i|^{-1/2} |E_e^j|^{-1/2} h_{E_e^j } \leq C$ and $|E_e^j | \geq C (1/\rho)^d h_{E_e^j}^{d}$, we obtain 
\begin{align} &\|\bm{w}_h|_{E_e^{i}}\|_{L^4(e)} \|(\bm{v} - \Pi_h \bm{v})|_{E_e^{j}}\|_{L^2(e)}\|\bm{\theta}_h|_{E_e^{j}}\|_{L^4(e)}  \nonumber \\
    &\leq Ch_{E_e^j}^{k-d/4}\|\bm{w}_h\|_{L^2(E_e^i)} |\bm{v}|_{H^{k+1}(\Delta_{E_e^j})} \|\bm{\theta}_h\|_{L^4(E_e^j)}. \label{eq:third_nonlinear_estimate_2}  
    \end{align}
Note that 
we have 
\begin{equation}
    \left(\sum_{e \in \Gamma_h} |\bm{v}|_{H^{k+1}(\Delta_{E_e^j)}}^4\right)^{1/4} \leq \left(\sum_{e \in \Gamma_h } |\bm{v}|_{H^{k+1}(\Delta_{E_e^j})}^2\right)^{1/2} \leq C |\bm{v} |_{H^{k+1}(\Omega)}.\nonumber  
\end{equation}
This bound along with similar bounds to \eqref{eq:third_nonlinear_estimate_2}, discrete H\"{o}lder's inequality, and \eqref{eq:discreter_poincare} show estimate \eqref{eq:second_estimate_nonlinear_term}. 
\end{proof}
With the bounds established in Lemma \ref{prep:bounds_nonlinear_terms} and Lemma \ref{lemma:bounds_nonlinear_terms_2}, we show a bound on $R_{\mathcal{C}} (\err)$. 

\begin{lemma} \label{lemma:bounds_Rnl}
Assume that $ \bm{u} \in L^{\infty}(0,T; H^{k+1}(\Omega)^d)$ and $\partial_t \bm{u} \in L^2(0,T; L^2(\Omega)^d)$. There exist positive constants $\gamma$ and $C$  independent of $h$,  $\tau$ and $\mu$   such that if  $\tau \leq \gamma  \mu $, the following bound holds. For $n \geq 1$,  
\begin{multline}
| R_{\mathcal{C}}(\err)| \leq   \frac{C}{\mu} \tau \int_{t^{n-1}}^{t^n} \|\partial_t \bm{u} \|^2 
+ \frac{C}{\mu}   h^{2k} + 
\frac{C}{\mu}   \|\bm{e}_h^{n-1} \|^2 +  \frac{\kappa \mu}{16}  \|\err \|^2_{\DG}   \\  
+ \frac{1}{16} \tau \sum_{e\in\Gamma_h} \frac{\tilde{\sigma}}{h_e} \|[\phi_h^n]\|^2_{L^2(e)} + \frac{1}{4}\tau \sum_{e\in\Gamma_h} \frac{\tilde{\sigma}}{h_e} \|[\phi_h^n- \phi_h^{n-1}]\|^2_{L^2(e)}   
+ \frac{\delta_{n,1}}{2} |b(\bm{e}_h^{0}, \Pi_h \bm{u}^1 \cdot \tilde{\bm{e}}_h^1)|.    \label{eq:bound_nonlinear_term}
\end{multline}
\end{lemma}
\begin{proof}
We follow a similar technique to \cite{girault2005splitting}. Since $\bm{u}^n$ belongs to $H^1_0(\Omega)^d$, we can write 
\[
a_\mathcal{C}(\bm{u}^{n}; \bm{u}^{n}, \bm{u}^n, \err)
=
a_\mathcal{C}(\bm{u}_h^{n-1}; \bm{u}^{n}, \bm{u}^n, \err).
\]
Therefore, we have
\begin{align*}
      R_{\mathcal{C}}(\err) 
     & = a_\mathcal{C}(\bm{u}_h^{n-1}; \bm{u}_h^{n-1}, \Pi_h \bm{u}^n - \bm{u}^n, \err) + a_\mathcal{C}(\bm{u}_h^{n-1}; \bm{e}_h^{n-1},  \bm{u}^n, \err) \nonumber \\ & \quad + a_\mathcal{C}(\bm{u}_h^{n-1};\Pi_h \bm{u}^{n-1} - \bm{u}^{n}, \bm{u}^n, \err)  = W^n_1 + W^n_2 + W^n_3. 
    \end{align*} 
With the help of \eqref{eq:split_ac2}, we write $W^n_1$ as follows. 
\begin{align*}
    W^n_1  
    & =  \mathcal{C}(\bm{e}_h^{n-1}, \Pi_h \bm{u}^n - \bm{u}^n, \err) + \mathcal{C}(\Pi_h \bm{u}^{n-1}, \Pi_h \bm{u}^n - \bm{u}^n, \err) \\ & \quad -  \mathcal{U}(\bm{u}_h^{n-1}; \bm{e}_h^{n-1}, \Pi_h \bm{u}^n - \bm{u}^n, \err) -  \mathcal{U}(\bm{u}_h^{n-1}; \Pi_h \bm{u}^{n-1}, \Pi_h \bm{u}^n - \bm{u}^n, \err).
\end{align*}
For $n=1$, with \eqref{eq:expanding_C} and \eqref{eq:third_nonlinear_term_0}, we have: 
\begin{equation*}
 \mathcal{C}(\bm{e}_h^{0}, \Pi_h \bm{u}^1 - \bm{u}^1, \tilde{\bm{e}}_h^1) \leq C\|\bm{e}_h^{0}\||\bm{u}^1|_{W^{1,3}(\Omega)} \|\tilde{\bm{e}}_h^1\|_{\DG} + \frac{1}{2}b(\bm{e}_h^0, (\Pi_h \bm{u}^1 - \bm{u}^1) \cdot \tilde{\bm{e}}_h^1). 
\end{equation*}
Since \eqref{eq:div_error_2} holds, we can apply \eqref{eq:second_estimate_nonlinear_term_C1} in Lemma \ref{prep:bounds_nonlinear_terms} to bound the first term in $W_1^n$ for $n \geq 2$. We also use estimates  \eqref{eq:second_estimate_nonlinear_term_C2}  for the second  term, \eqref{eq:second_estimate_nonlinear_term} for the third term, and \eqref{eq:second_estimate_nonlinear_term_U2} along with \eqref{eq:discreter_poincare} for the  fourth term. For $n \geq 1$, we obtain
\begin{align*}
 &W_1^n \leq  C \|\bm{e}_h^{n-1}\| \,  \vertiii{\bm{u}^n}  \,  \| \err \|_{\DG} 
+ C\tau \left( \sum_{e \in \Gamma_h} \frac{\tilde{\sigma}}{h_e} \| [\phi_h^{n-1}] \|_{L^2(e)}^2 \right)^{1/2}
\, \vertiii{\bm{u}^n} \, \|\err\|_{\DG} \nonumber \\ 
& + C  h^{k-d/4} \vert \bm{u}^n \vert_{H^{k+1}(\Omega)}\, \| \bm{e}_h^{n-1}\| \, \|\err\|_{\DG} 
+ C h^{k} \vert \bm{u}^n\vert_{H^{k+1}(\Omega)} \,\vertiii{\bm{u}^{n-1}} \,  \|\err\|_{\DG} \\ 
&  +  \frac{1}{2} \delta_{n,1} b(\bm{e}_h^{0}, (\Pi_h \bm{u}^1 - \bm{u}^1) \cdot \tilde{\bm{e}}_h^1).
\end{align*}
Further, for $k \geq 1$, with Morrey's inequality and a Sobolev embedding (see Theorem 1.4.6 in \cite{brenner2007mathematical} and Theorem 2 in section 5.6.1 \cite{evans2010partial}), we have for $n \geq 0$: 
\begin{equation}
\vertiii{\bm{u}^n} \leq C \Vert \bm{u}^n \Vert_{H^{k+1}(\Omega)} \leq C \Vert \bm{u} \Vert_{L^{\infty}(0,T;H^{k+1}(\Omega))}. \label{eq:bd_vertiii}
\end{equation}
With the help of \eqref{eq:bd_vertiii}, the assumption that $\bm{u}\in L^\infty(0,T;H^{k+1}(\Omega)^d)$, and Young's inequality, we obtain for constants $C$ and $C_1$: 
\begin{align}
&W^n_1  \leq  \frac{C}{\kappa\mu}   \| \bm{e}_h^{n-1}\|^2 +  \frac{C}{\kappa\mu}   h^{2k} + \left( \frac{\kappa\mu }{96} + C_1\tau \right) \| \err \|^2_{\DG} + \frac{\tau}{32} \sum_{e\in\Gamma_h} \frac{\tilde{\sigma}}{h_e} \| [\phi_h^n] \|^2_{L^2(e)}   \nonumber \\ &  + \frac{\tau}{8} \sum_{e\in\Gamma_h} \frac{\tilde{\sigma}}{h_e} \| [\phi_h^n - \phi_h^{n-1}] \|^2_{L^2(e)} +  \frac{1}{2} \delta_{n,1} b(\bm{e}_h^{0}, (\Pi_h \bm{u}^1 - \bm{u}^1) \cdot \tilde{\bm{e}}_h^1) \nonumber . 
\end{align}
With \eqref{eq:ident_cont_u} and \eqref{eq:first_nonlinear_estimate_0}, we have 
\begin{equation}
a_{\mathcal{C}}(\bm{u}_h^{0}; \bm{e}_h^0, \bm{u}^1, \tilde{\bm{e}}_h^1) \leq C\|\bm{e}_h^{0}\||\bm{u}^1|_{W^{1,3}(\Omega)} \|\tilde{\bm{e}}_h^1\|_{\DG} + \frac{1}{2}b(\bm{e}_h^0,  \bm{u}^1 \cdot \tilde{\bm{e}}_h^1). \nonumber 
\end{equation}
Since \eqref{eq:div_error_2} holds, we can apply \eqref{eq:first_estimate_nonlinear_form}  to bound $W^n_2$ for $n \geq 2$. Thus, we have 
\begin{multline*}
W^n_2\leq C  \|\bm{e}_h^{n-1}\| \, \vertiii{\bm{u}^n} \, \|\err\|_{\DG}
+ C\tau\left( \sum_{e \in \Gamma_h} \frac{\tilde{\sigma}}{h_e} \| [\phi_h^{n-1}] \|_{L^2(e)}^2 \right)^{1/2}
\,  \vertiii{\bm{u}^n}\,  \|\err\|_{\DG}  \\ + \frac{1}{2} \delta_{n,1} b(\bm{e}_h^{0},  \bm{u}^1 \cdot \tilde{\bm{e}}_h^1).
\end{multline*}
With \eqref{eq:bd_vertiii} and Young's inequality, we obtain  for constants $C$ and $C_2$
\begin{align}
        W^n_2 
       & \leq 
        \frac{C}{\kappa\mu}  \|\bm{e}_h^{n-1}\|^2 + \left( \frac{\kappa \mu}{96} + C_2 \tau \right)\|\err\|_{\DG}^2 +  \frac{\tau}{32}\sum_{e\in \Gamma_h} \frac{\tilde{\sigma}}{h_e} \|  [\phi_h^{n}]  \|_{L^2(e)}^2  \nonumber \\  & \quad    + \frac{\tau}{8}\sum_{e\in\Gamma_h} \frac{\tilde{\sigma}}{h_e} \| [\phi_h^n - \phi_h^{n-1}] \|^2_{L^2(e)} + \frac{1}{2} \delta_{n,1}b(\bm{e}_h^{0},  \bm{u}^1 \cdot \tilde{\bm{e}}_h^1) \nonumber .
    \end{align}
    To bound $W^n_3$, we split it into  
    \begin{equation} \label{eq:splitW3} 
    W^n_3 = a_\mathcal{C}(\bm{u}_h^{n-1};\Pi_h \bm{u}^{n-1} - \bm{u}^{n-1}, \bm{u}^n, \err) +  a_\mathcal{C}(\bm{u}_h^{n-1};\bm{u}^{n-1} - \bm{u}^{n}, \bm{u}^n, \err). \nonumber 
    \end{equation}
   Since $b(\bm{u}^{n-1}, q_h) = 0, \forall q_h \in M_h, \forall n \geq 1 $, we apply \eqref{eq:third_estimate_nonlinear_term} and the assumption that $\bm{u} \in L^{\infty}(0,T;H^{k+1}(\Omega)^d)$ to bound the first term. 
    \begin{align}
         |a_\mathcal{C}(\bm{u}_h^{n-1}; \Pi_h \bm{u}^n - \bm{u}^n, \bm{u}^n, \err)| & \leq Ch^{k+1} \vert \bm{u}^n\vert_{H^{k+1}(\Omega)} \vertiii{\bm{u
         }^n}\| \err \|_{\DG} \nonumber  \\ 
    & \leq  \frac{C}{\kappa\mu}  h^{2k+2}  + \frac{\kappa \mu}{96} \|\err \|_{\DG}^2.  \nonumber 
    \end{align}
Since $\bm{u}^{n-1}$ and $\bm{u}^n$ belong to $H_0^1(\Omega)^d$ and their divergence is zero, we have  
    \begin{align}
        a_\mathcal{C}(\bm{u}_h^{n-1};\bm{u}^{n-1} - \bm{u}^{n}, \bm{u}^n, \err) = \sum_{E \in \mesh_h} \int_{E} ((\bm{u}^{n-1} - \bm{u}^n) \cdot \nabla \bm{u}^n) \cdot  \err. \nonumber  
    \end{align}
    By H\"{o}lder's inequality, \eqref{eq:discreter_poincare}, and \eqref{eq:bd_vertiii}, we obtain  
    \begin{align}
         | a_\mathcal{C}(\bm{u}_h^{n-1};\bm{u}^{n-1} - \bm{u}^{n}, \bm{u}^n, \err) | &\leq  \| \bm{u}^{n-1} - \bm{u}^n \| \vert \bm{u}^n \vert_{W^{1,3}(\Omega)} \|\err \|_{\DG} \nonumber \\
        & \leq  \frac{C}{\kappa\mu}  \|\bm{u}^{n-1} - \bm{u}^n \|^2 |\bm{u}^n|^2_{W^{1,3}(\Omega)}+ \frac{\kappa\mu}{96} \| \err \|_{\DG}^2  \nonumber  \\ 
        &\leq \frac{C}{\kappa \mu} \tau  \int_{t^{n-1}}^{t^n} \|\partial_t \bm{u} \|^2 + \frac{\kappa\mu}{96} \| \err \|_{\DG}^2. \nonumber 
    \end{align}
    The result follows by combining all the above bounds, and by assuming  $(C_1+C_2) \tau \leq \kappa \mu/96$.
    \end{proof}
We are now ready to present and prove the main convergence result. 
\begin{theorem}\label{prep:first_error_estimate_velocity}
 Assume that $\sigma \geq M_{k-1}^2/d$,  $\tilde{\sigma} \geq 4 \tilde{M}_k^2$, and $\delta \leq \kappa/(2d)$. There exists a positive constant $\gamma_1$, independent of $h,\tau$ and $\mu$, such that if $\tau \leq \gamma_1 \mu$,  the following estimate holds.  For $0 \leq m \leq N_T$, 
    \begin{align}
        \|\bm{u}_h^{m} - \bm{u}^m\|^2 +  \frac{\kappa \mu}{4} \tau \sum_{n=1}^{m}  \|\bm{v}_h^n - \bm{u}^n\|_{\DG}^2 & \leq C_\mu \left( 1+  \mu + \frac{1}{\mu} \right)(  \tau + h^{2k}). \label{eq:first_error_estimate_0} 
        \end{align}
         The above estimate holds under the following regularity assumptions:

         $\bm{u} \in L^{\infty}(0,T; H^{k+1}(\Omega)^d)$, $\partial_t \bm{u} \in L^2(0,T ; H^{k}(\Omega)^d)$, $\partial_{tt} \bm{u} \in L^2(0,T; L^2(\Omega)^d)$, and $p \in L^{\infty}(0,T; H^k(\Omega))$. 
        \end{theorem}
        \begin{proof}
        Let $\bm{\theta}_h = \err$ in \eqref{eq:first_err_eq}. Using the positivity property of $a_\mathcal{C}$ \eqref{eq:cpositivity} and the coercivity property of $a_\epsilon$ \eqref{eq:coercivity_a_epsilon},  we have 
        \begin{multline}
        \frac{1}{2}(\|\err\|^2 - \|\bm{e}_h^{n-1}\|^2 +  \|\err - \bm{e}_h^{n-1}\|^2)+ \kappa \mu \tau  \|\err\|_{\DG}^2 + \tau R_{\mathcal{C}} (\err) \leq \tau b(\err, p_h^{n-1} - p^n) \\  - \tau \mu a_\epsilon(\Pi_h \bm{u}^n  - \bm{u}^n, \err ) + R_t(\err). \nonumber 
        \end{multline}
        Substituting \eqref{eq:error_diff} in the above expression, we obtain 
        \begin{equation}
            \begin{split}
                \frac{1}{2}(&\|\err\|^2 - \|\bm{e}_h^{n-1}\|^2 +  \|\bm{e}_h^n - \bm{e}_h^{n-1}\|^2 )  + \frac{\tau^2}{2}(\| \nabla_h \phi_h^n \|^2 + \|\bm{G}_h([\phi_h^n])\|^2)  \\ & + \kappa \mu \tau \|\err\|_{\DG}^2  + \frac{\tau^2}{2}(A^n_1 - A^n_2) + \frac{\tau^2}{2} \sum_{e\in\Gamma_h }  \frac{\tilde{\sigma}}{h_e} \| [\phi_h^n - \phi_h^{n-1}]\|_{L^2(e)}^2  \leq - \tau  R_{\mathcal{C}} (\err)\\  & + \tau b(\err, p_h^{n-1} - p^n) -  \tau \mu a_{\epsilon}(\Pi_h \bm{u}^n - \bm{u}^n ,\err )  +  R_t(\err)  + \frac{\tau^2}{2} \|\bm{G}_h([\phi_h^n - \phi_h^{n-1}])\|^2 \\  &  + \tau^2 (\nabla_h \phi_h^n, \bm{G}_h([\phi_h^n])) - \delta_{n,1} \tau b(\bm{e}_h^{0}, \phi_h^1) .\label{eq:conv_1}
            \end{split}
        \end{equation}
We begin by handling the fifth and sixth terms. With \eqref{eq:lift_prop_g} and the assumption that $\tilde{\sigma} \geq 2\tilde{M}_{k}^2$, we have: 
\begin{multline}
 \|\bm{G}_h([\phi_h^n - \phi_h^{n-1}])\|^2  \leq  \sum_{e\in \Gamma_h} \frac{\tilde{M}_k^2}{h_e} \| [\phi_h^{n} - \phi_h^{n-1}]\|^2_{L^2(e)} \\ \leq  \frac{1}{2}\sum_{e\in \Gamma_h} \frac{\tilde{\sigma}}{h_e} \| [\phi_h^{n} - \phi_h^{n-1}]\|^2_{L^2(e)} .  \label{eq:bound_A3}
\end{multline}
    Using Cauchy-Schwarz's inequality,  Young's inequality, \eqref{eq:lift_prop_g}, and the assumption that $\tilde{\sigma} \geq 4\tilde{M}_k^2$, we obtain 
        \begin{multline}
         |(\nabla_h \phi_h^n , \bm{G}_h([\phi_h^n ]))|  
             \leq \frac{1}{4}\| \nabla_h \phi_h^n\|^2 + \| \bm{G}_h([\phi_h^n])\|^2 \\ \leq \frac{1}{4}\| \nabla_h \phi_h^n\|^2 + \sum_{e\in\Gamma_h} \frac{\tilde{M}_k^2}{h_e} \|[\phi_h^n] \|^2_{L^2(e)}    \leq  \frac{1}{4}\| \nabla_h \phi_h^n\|^2 + \frac{1}{4} \sum_{e\in\Gamma_h} \frac{\tilde{\sigma}}{ h_e} \|[\phi_h^n] \|^2_{L^2(e)}. 
             \label{eq:bound_A4}
        \end{multline}
         Next, let $\bm{\theta}_h = \bm{e}_h^n$ in \eqref{eq:sec_err_eq} and use \eqref{eq:def_b_lift_2}. 
        \begin{equation}
        \frac12 \left( \|\bm{e}_h^n\|^2 - \|\err\|^2 \right) + \frac12 \|\bm{e}_h^n - \err \|^2 = \tau b(\bm{e}_h^{n}, \phi_h^n). 
        \end{equation}
        Using \eqref{eq:div_error_2}, we obtain 
        \begin{equation}
        \frac12 \left( \|\bm{e}_h^n\|^2 - \|\err\|^2 \right) + \frac12 \|\bm{e}_h^n - \err \|^2 + \tau^2 \sum_{e\in\Gamma_h} \frac{\tilde{\sigma}}{h_e} \|[\phi_h^n] \|_{L^2(e)}^2 = \tau^2 \|\bm{G}_h([\phi_h^n]) \|^2. \label{eq:conv_2_step1}
        \end{equation}
        Similar to the stability proof, we let $\bm
        {\theta}_h = \bm{e}_h^{n} - \err$ in \eqref{eq:sec_err_eq} and use \eqref{eq:div_error}: 
        \begin{multline}
        \| \bm{e}_h^n - \err\|^2 = \tau b(\bm{e}_h^n - \err, \phi_h^n)= \tau^2 \aelip(\phi_h^n,\phi_h^n) \\  -  \tau^2 \sum_{e\in\Gamma_h} \frac{\tilde{\sigma}}{h_e} \|[\phi_h^n] \|_{L^2(e)}^2 + \tau^2 \|\bm{G}_h([\phi_h^n]) \|^2. 
        \end{multline}
        Hence, \eqref{eq:conv_2_step1} reads 
        \begin{multline}
            \frac12 \left( \|\bm{e}_h^n\|^2 - \|\err\|^2 \right) + \frac{\tau^2}{2} \aelip(\phi_h^n,\phi_h^n) \\ + \frac{\tau^2}{2} \sum_{e\in \Gamma_h} \frac{\tilde{\sigma}}{h_e} \|[\phi_h^n] \|^2_{L^2(e)} = \frac{\tau^2}{2} \|\bm{G}_h([\phi_h^n])\|^2.  \label{eq:conv_2}
        \end{multline}
        We add \eqref{eq:conv_2} to \eqref{eq:conv_1}. With bounds \eqref{eq:bound_A3}  and  \eqref{eq:bound_A4}, we obtain
        \begin{equation}
            \begin{aligned}                
                &\frac{1}{2}(\|\bm{e}_h^n\|^2 - \|\bm{e}_h^{n-1}\|^2+ \|\bm{e}_h^n - \bm{e}_h^{n-1}\|^2)  + \frac{\tau^2}{4} \vert \phi_h^n \vert_{\DG}^2  +\frac{\tau^2}{2}\aelip(\phi_h^n, \phi_h^n) + \kappa \mu \tau \|\err\|_{\DG}^2 \\ & + \frac{\tau^2}{2}(A^n_1 - A^n_2) + \frac{\tau^2}{4} \sum_{e\in\Gamma_h }  \frac{\tilde{\sigma}}{h_e} \| [\phi_h^n - \phi_h^{n-1}]\|_{L^2(e)}^2 \leq - \tau R_{\mathcal{C}}(\err) + \tau b(\err, p_h^{n-1} - p^n) \\ & -\tau \mu a_{\epsilon}(\Pi_h \bm{u}^n - \bm{u}^n ,\err )  +  R_t(\err) - \delta_{n,1} \tau  b(\bm{e}_h^{0}, \phi_h^1). \label{eq:conv_add}
            \end{aligned}
        \end{equation}
    We now focus on the second term on the right hand side of \eqref{eq:conv_add}. We write 
        $$  b(\err, p_h^{n-1} - p^n) = b(\err, p_h^{n-1}) - b(\err, \pi_h p^n) + b(\err, \pi_h p^n - p^n).  $$ 
         Since the operator $\pi_h$ preserves cell averages and $p^n$ satisfies \eqref{eq:zero_avg}, we know that $\pi_h p^n \in M_{h0}$.  By \eqref{eq:error_pressure_correction}, continuity of $\aelip(\cdot, \cdot)$ on $M_h \times M_h$ see \eqref{eq:continuity_aellip}, Young's inequality, and \eqref{eq:bd_phN}, we obtain 
        \begin{multline} \label{eq:bounding_b_zh}
           | -b(\err, \pi_h p^n)| = \tau| \aelip(\phi_h^n, \pi_h p^n)| \leq C\tau \vert \phi^n_h\vert_{\DG} \vert \pi_h p^n\vert_{\DG}\\ \leq \frac{\tau}{8} |\phi_h^n|_{\DG}^2 + C\tau |p^n |_{H^1(\Omega)}^2. 
        \end{multline}
        Since $\nabla_h \cdot \err \in M_h$, by the definition of $\pi_h p^n$ we have:  
        \begin{align*}
      b(\err, \pi_h p^n - p^n ) = - \sum_{e\in \Gamma_h \cup \partial \Omega} \int_{e} \{ \pi_h p^n - p^n \} [\err]\cdot \bm{n}_e. 
         \end{align*}
    By using trace inequalities and \eqref{eq:l2_proj_approximation}, we obtain 
        \begin{align}
         |b(\err, \pi_h p^n - p^n ) |  \leq  \frac{C}{ \mu }  h^{2k} \vert p^n\vert_{H^{k}(\Omega)}^2 + \frac{\kappa \mu}{16} \|\err \|_{\DG}^2.\label{eq:bd_b2}
        \end{align}
        Thus, with bounds \eqref{eq:bounding_b_zh} and \eqref{eq:bd_b2}, \eqref{eq:conv_add} becomes: 
        \begin{equation}
            \begin{aligned}                
                &\frac{1}{2}(\|\bm{e}_h^n\|^2 - \|\bm{e}_h^{n-1}\|^2 +\|\bm{e}_h^n - \bm{e}_h^{n-1}\|^2)  + \frac{\tau^2}{8} \vert \phi_h^n \vert_{\DG}^2  +\frac{\tau^2}{2}\aelip(\phi_h^n, \phi_h^n) + \kappa  \mu \tau  \|\err\|_{\DG}^2   \\ & + \frac{\tau^2}{2}(A^n_1 - A^n_2) + \frac{\tau^2}{4} \sum_{e\in\Gamma_h }  \frac{\tilde{\sigma}}{h_e} \| [\phi_h^n - \phi_h^{n-1}]\|^2_{L^2(e)} \leq - \tau R_{\mathcal{C}} (\err) + \tau b(\err, p_h^{n-1}) \\ &   + C\tau^2 \vert p^n\vert _{H^1(\Omega)}^2 + \frac{C}{ \mu} \tau  h^{2k} \vert p^n\vert_{H^{k}(\Omega)}^2 - \tau \mu a_{\epsilon}(\Pi_h \bm{u}^n - \bm{u}^n ,\err)   +  \frac{\kappa \mu }{16} \tau  \|\err \|_{\DG}^2 \\ & + R_t(\err) -\delta_{n,1} \tau  b(\bm{e}_h^{0}, \phi_h^1).     \label{eq:error_eq_semi_final}  
            \end{aligned}
        \end{equation}
    Note that by \eqref{eq:def_pi_h}, $b(\err, p_h^{n-1}) = b(\bm{v}_h^n ,p_h^{n-1})$. Hence, from \eqref{eq:rhs_final_form}, we have 
        \begin{multline}
            b(\err, p_h^{n-1}) = -\frac{\tau}{2} \left( a_{\mathrm{ellip}}(\xi_h^n, \xi_h^n ) -  a_{\mathrm{ellip}}(\xi_h^{n-1},\xi_h^{n-1}) - a_{\mathrm{ellip}}(\phi_h^n,\phi_h^n)\right) \\ - \frac{1}{2\delta \mu} \left(\|S_h^n \|^2 - \|S_h^{n-1}\|^2 - \|S_h^n - S_h^{n-1}\|^2 \right).
            \label{eq:rhs_final_form_err} 
        \end{multline}  
        With this expression, \eqref{eq:error_eq_semi_final} becomes  
            \begin{equation}
                \begin{aligned}
            &\frac{1}{2}(\|\bm{e}_h^n\|^2 - \|\bm{e}_h^{n-1}\|^2 + \|\bm{e}_h^n - \bm{e}_h^{n-1}\|^2 ) + \frac{\tau^2}{2}(\aelip(\xi_h^n, \xi_h^n) - \aelip(\xi_h^{n-1},\xi_h^{n-1})) \\ &+ \frac{\tau^2}{8} \vert \phi_h^n \vert_{\DG}^2 +  \kappa  \mu \tau \|\err\|_{\DG}^2 + \frac{\tau^2}{4} \sum_{e\in\Gamma_h }  \frac{\tilde{\sigma}}{h} \| [\phi_h^n - \phi_h^{n-1}]\|^2_{L^2(e)}+ \frac{\tau^2}{2}(A^n_1 - A^n_2)   \\ &+ \frac{\tau}{2\delta\mu}(\|S_h^n\|^2- \|S_h^{n-1}\|^2)   \leq    - \tau \mu a_{\epsilon}(\Pi_h \bm{u}^n - \bm{u}^n , \err )- \tau  R_{\mathcal{C}} (\err)+ C\tau^2 \vert p^n\vert _{H^1(\Omega)}^2  \\ 
&    + \frac{\tau}{2\delta \mu} \|S_h^n -S_h^{n-1}\|^2 +  \frac{C}{\mu} \tau h^{2k} \vert p^n\vert_{H^{k}(\Omega)}^2 + \frac{\kappa \mu}{16} \tau\|\err \|_{\DG}^2 + R_t(\err) \\
& -\delta_{n,1} \tau b(\bm{e}_h^{0}, \phi_h^1).   \label{eq:error_eq_semi_final_2}
            \end{aligned}
        \end{equation} 
        We now consider $R_t(\err)$ (see \eqref{eq:defRt}). Using Cauchy-Schwarz's inequality, Young's inequality and  \eqref{eq:discreter_poincare},  we obtain 
        \begin{align}
            |R_t(\err)| & =\left| \tau((\partial_t \bm{u})^n, \err ) +  (\Pi_h\bm{u}^{n-1} - \Pi_h \bm{u}^{n},\err) \right|  \nonumber   \\
        & \leq \frac{C}{\mu} \tau^{-1} \| \tau (\partial_t \bm{u})^n - (\bm{u}^n - \bm{u}^{n-1}) \|^2 \nonumber \\ & \quad + \frac{C}{\mu}\tau^{-1} \|(\Pi_h \bm{u}^{n-1} - \bm{u}^{n-1}) - (\Pi_h \bm{u}^n - \bm{u}^n )\|^2  +  \frac{\kappa  \mu}{16}\tau \| \err \|^2_{\DG}. \nonumber  
        \end{align}
        With a Taylor expansion and approximation property \eqref{eq:approximation_prop_1}, we have
        \begin{multline}
             |R_t(\err)|   \leq  \frac{C}{\mu}   \tau^2 \int_{t^{n-1}}^{t^n} \| \partial_{tt} \bm{u}\|^2  + \frac{C}{\mu}   \int_{t^{n-1}}^{t^n} \|\partial_t  (\Pi_h\bm{u} - \bm{u})\|^2  + \frac{\kappa  \mu}{16} \tau \| \err \|^2_{\DG}  \\
             \leq \frac{C}{\mu}    \tau^2 \int_{t^{n-1}}^{t^n} \| \partial_{tt} \bm{u}\|^2   
     + \frac{C}{\mu} h^{2k} \int_{t^{n-1}}^{t^n} \vert \partial_t  \bm{u}\vert_{H^{k}(\Omega)}^2 + \frac{\kappa  \mu}{16} \tau\| \err \|^2_{\DG}. \label{eq:bound_on_Rt}  
        \end{multline}
        With the definition of $S^n_h$ \eqref{eq:def_S} and \eqref{eq:equivalent_def_Pi_h}, we have 
        \begin{equation}
            S_h^n - S_h^{n-1} = \delta \mu (\nabla_h \cdot \err - R_h([\err])). 
        \end{equation}
        With the assumptions that $\delta \leq \kappa/(2d)$ and $\sigma \geq M_{k-1}^2/d$ and with \eqref{eq:lift_prop_r}, we have 
        \begin{align}
        \frac{1}{2\delta\mu}\| S_h^n - S_h^{n-1}\|^2 &\leq \delta \mu \| \nabla_h \cdot \err \|^2 +  \delta \mu  \| R_h([\err])\|^2  \nonumber \\
            & \leq \delta \mu d \| \nabla_h \err\|^2 + \delta \mu M^2_{k-1} \sum_{e\in\Gamma_h \cup \partial \Omega} h_e^{-1}\|[\err]\|_{L^2(e)}^2 \nonumber \\ 
            & \leq \frac{\kappa  \mu}{2}  \| \nabla_{h} \err \|^2 + \frac{\kappa \mu }{2}  \sum_{e\in \Gamma_h \cup \partial \Omega} \sigma h_e^{-1} \| [\err] \|^2_{L^2(e)}. \label{eq:bd_tildeS}
        \end{align}
        To handle $a_{\epsilon} (\Pi_h \bm{u}^n - \bm{u}^n, \err )$, we write: 
        \begin{align*}
         & a_{\epsilon} (\Pi_h \bm{u}^n - \bm{u}^n, \err ) =  \sum_{E \in \mesh_h} \int_{E} \nabla(\Pi_h \bm{u}^n - \bm{u}^n) : \nabla \err   \\ & - \sum_{e\in\Gamma_h \cup \partial \Omega} \int_{e} \{\nabla(\Pi_h \bm{u}^n - \bm{u}^n )\} \bm{n}_e\cdot [\err] 
  + \epsilon \sum_{e\in \Gamma_h \cup \partial \Omega} \int_e \{ \nabla \err\}\bm{n}_e \cdot [\Pi_h \bm{u}^n - \bm{u}^n] \\ &+ \sum_{e\in \Gamma_h \cup \partial \Omega} \frac{\sigma}{h_e}\int_e [\Pi_h \bm{u}^n - \bm{u}^n]\cdot [\err]  = Q_1 +Q_2 +Q_3 +Q_4.
    \end{align*}
    The terms $Q_1,Q_3$ and $Q_4$ are handled via standard arguments. By using trace inequalities, we have 
        \begin{equation}
        |Q_1| +|Q_3| +|Q_4| \leq \frac{\kappa}{32}  \|\err \|^2_{\DG} + C h^{2k} \vert \bm{u}^n \vert_{H^{k+1}(\Omega)}^2. \label{eq:bd_W_1_W_3}
        \end{equation} 
        The term $Q_2$ is more delicate. We follow an approach similar to the one in Chapter 6 of \cite{riviere2008discontinuous} and \cite{girault2005discontinuous}. By Cauchy-Schwarz's inequality, we have 
        \begin{align}
        |Q_2| \leq Ch^{1/2}\left(\sum_{e\in\Gamma_h \cup \partial \Omega} \|\{ \nabla(\Pi_h \bm{u}^n - \bm{u}^n)\} \|_{L^2(e)}^2\right)^{1/2} \| \err\|_{\DG}. \label{eq:bd_W2}
        \end{align}
Consider one element $E_e^1$ adjacent to $e$. By the trace theorem, we have 
\begin{multline}
\| \nabla(\Pi_h \bm{u}^n - \bm{u}^n) \vert_{E_e^1}\|_{L^2(e)} \leq Ch_{E_e^1}^{-1/2} (\|\nabla(\Pi_h \bm{u}^n - \bm{u}^n) \|_{L^2(E_e^1)} \\ + h_{E_e^1} \|\nabla^2 (\Pi_h \bm{u}^n - \bm{u}^n)\|_{L^2(E_e^1)}). \label{eq:edge_grad}
\end{multline}
Let $\tilde{\bm{u}}^n \in \bm{X}_h$ be an interpolant of $\bm{u}^n$, see Theorem 2.6 in \cite{riviere2008discontinuous}, satisfying:
\begin{equation}
\|\nabla (\bm{u}^n - \tilde{\bm{u}}^n )\|_{L^2(E)} + h_E \|\nabla^2(\bm{u}^n -\tilde{\bm{u}}^n)\|_{L^2(E)} \leq Ch_{E}^{k}\vert \bm{u}^n \vert_{H^{k+1}(E)}, \quad \forall E \in \mesh_h.  \nonumber
\end{equation}
By the triangle inequality, an inverse estimate, the above bound and \eqref{eq:approximation_prop_2}, we obtain: 
\begin{align*}
& \|\nabla^2 (\Pi_h \bm{u}^n - \bm{u}^n)\|_{L^2(E_e^1)} \leq \|\nabla^2 (\Pi_h \bm{u}^n - \tilde{\bm{u}}^n)\|_{L^2(E_e^1)} + \| \nabla^2 (\tilde{\bm{u}}^n - \bm{u}^n)\|_{L^2(E_e^1)} \\ 
& \leq  Ch_{E_e^1}^{-1} \| \nabla (\Pi_h \bm{u}^n - \tilde{\bm{u}}^n)\|_{L^2(E_e^1)} +C h_{E_e^1}^{k-1}\vert \bm{u}^n \vert_{H^{k+1}(E_e^1)}
\leq C h_{E_e^1}^{k-1} \vert \bm{u}^n \vert_{H^{k+1}(E_e^1)}.
\end{align*}
Hence, \eqref{eq:edge_grad} becomes 
\begin{align}
\| \nabla(\Pi_h \bm{u}^n - \bm{u}^n)\vert_{E_e^1}\|_{L^2(e)} \leq Ch_{E_e^1}^{k-1/2} \vert \bm{u}^n \vert_{H^{k+1}(E_e^1)}. 
\end{align}
The term involving the neighboring element $E_e^2$ is handled similarly.
Thus, we can conclude that 
\begin{equation}
|Q_2| \leq \frac{\kappa}{32}  \|\err \|^2_{\DG} + C h^{2k} \vert \bm{u}^n \vert_{H^{k+1}(\Omega)}^2. \label{eq:bd_W_2} \end{equation}
We employ the above bounds on $Q_1$--$Q_4$, the bound on the nonlinear terms $| R_{\mathcal{C}} (\err)|$ \eqref{eq:bound_nonlinear_term}, \eqref{eq:bd_tildeS}, and the bound on $|R_t(\err)|$ \eqref{eq:bound_on_Rt} in \eqref{eq:error_eq_semi_final_2}. With these bounds and the assumptions that $\bm{u} \in L^{\infty}(0,T;H^{k+1}(\Omega)^d)$ and $p \in L^{\infty}(0,T;H^k(\Omega))$, \eqref{eq:error_eq_semi_final_2} becomes:  
\begin{equation}
\begin{aligned}
    &\frac{1}{2}(\|\bm{e}_h^n\|^2 - \|\bm{e}_h^{n-1}\|^2 + \|\bm{e}_h^n - \bm{e}_h^{n-1}\|^2 ) + \frac{\tau^2}{2}(\aelip(\xi_h^n, \xi_h^n) - \aelip(\xi_h^{n-1}, \xi_h^{n-1}))\\& + \frac{\tau^2}{16} \vert \phi_h^n \vert_{\DG}^2 + \frac{\kappa  \mu}{4} \tau \|\err\|_{\DG}^2 +  \frac{\tau^2}{2}(A^n_1 - A^n_2) + \frac{\tau}{2\delta\mu}(\|S_h^n\|^2- \| S_h^{n-1}\|^2)   \\ &  \leq  \frac{C}{\mu}  \tau^2 \int_{t^{n-1}}^{t^n} (\|\partial_t \bm{u} \|^2 + \|\partial_{tt} \bm{u} \|^2) +  \frac{C}{\mu} h^{2k} \int_{t^{n-1}}^{t^n} \vert \partial_t \bm{u} \vert^2_{H^{k}(\Omega)}   +  C\left( \mu + \frac{1}{\mu}\right)\tau h^{2k} \\ & + \frac{C}{\mu} \tau \|\bm{e}_h^{n-1}\|^2 + C\tau^2  -\delta_{n,1} \tau  b(\bm{e}_h^{0}, \phi_h^1) + \frac12\delta_{n,1}\tau | b(\bm{e}_h^{0}, \Pi_h \bm{u}^1 \cdot \tilde{\bm{e}}_h^1)|. \label{eq:error_eq_semi_final_3}
    \end{aligned}
\end{equation} 
We sum \eqref{eq:error_eq_semi_final_3} from $n = 1$ to $n = m$, multiply by $2$, and use the regularity assumptions. We  obtain 
         \begin{multline}
             \|\bm{e}_h^m\|^2 + \sum_{n=1}^m \|\bm{e}_h^n - \bm{e}_h^{n-1} \|^2 + \tau^2\aelip(\xi_h^m, \xi_h^m) + \frac{\tau^2}{8} \sum_{n=1}^m \vert \phi_h^n \vert^2_{\DG}+\frac{\kappa  \mu}{2} \tau   \sum_{n=1}^{m}  \|\err\|_{\DG}^2  \\  +\tau^2 \sum_{n=1}^{m} (A^n_1 - A^n_2)  + \frac{\tau}{\delta \mu} \|S^m_h\|^2 \leq  C\left(1 + \mu+\frac{1}{\mu}\right)( \tau + h^{2k}) +  \frac{C}{\mu} \tau \sum_{n=0}^{m-1} \| \bm{e}_h^n \|^2 \\ +\|\bm{e}_h^0\|^2 + 2\tau|b(\bm{e}_h^0, \phi_h^1)| + \tau |b(\bm{e}_h^0, \Pi_h \bm{u}^1 \cdot \tilde{\bm{e}}_h^1) |.
             \label{eq:error_eq_semi_final_4}
         \end{multline}  
        We recall that  $\phi_h^{0} = 0$ and the definitions of $A_1^n$ and $A_2^n$ (see \eqref{eq:def_A1} - \eqref{eq:def_A2}). We use \eqref{eq:lift_prop_g} with the assumption that $\tilde{\sigma} \geq  \tilde{M}_k^2$ to obtain 
        \begin{align}
        \sum_{n=1}^{m} (A^n_1 - A^n_2) = &\sum_{e\in\Gamma_h} \frac{\tilde{\sigma}}{h_e} \|[\phi_h^{m}]\|^2_{L^2(e)}  - \|\bm{G}_h([\phi_h^{m}])\|^2 \nonumber  \\ 
         &\geq  \sum_{e\in \Gamma_h} \frac{\tilde{\sigma} - \tilde{M}_k^2}{h_e}\|[\phi_h^{m}]\|^2_{L^2(e)} \geq 0.  \nonumber
        \end{align}
        With \eqref{eq:def_b_lift_2}, \eqref{eq:lift_prop_g}, and approximation properties we have 
        \begin{equation*}
            |b(\bm{e}_h^0, \phi_h^1)| \leq C\|\bm{e}_h^0\||\phi_h^1|_{\DG} \leq C\tau^{-1}h^{2k+2}|\bm{u}^0|_{H^{k+1}(\Omega)}^2 + \frac{\tau}{16} |\phi_h^1|_{\DG}^2. 
        \end{equation*}
        With H\"{o}lder's inequality, trace inequality, stability of the interpolant, and \eqref{eq:discreter_poincare}, we obtain
        \begin{multline*}      
            |b(\bm{e}_h^0, \Pi_h \bm{u}^1 \cdot \tilde{\bm{e}}_h^1)| \leq C\| \nabla_h \bm{e}_h^0 \| \|\Pi_h \bm{u}^1\|_{L^3(\Omega)}\|\tilde{\bm{e}}_h^1\|_{\DG}  + \\ 
           C \sum_{e \in \Gamma_h \cup \partial \Omega} \sum_{i=1}^2 \|\Pi_h \bm{u}^1 \|_{L^3(E_e^i)} \|\tilde{\bm{e}}_h^1\|_{L^6(E_e^i)}h_{E_e^i}^{-1/2} \|[\bm{e}_h^0]\|_{L^2(e)} \\ \leq C \|\bm{e}_h^0\|_{\DG} \|\bm{u}^1\|_{W^{1,3}(\Omega)}\|\tilde{\bm{e}}_h^1\|_{\DG} \leq \frac{C}{\kappa \mu}h^{2k}\|\bm{u}^1\|_{H^{k+1}(\Omega)}^2 + \frac{\kappa \mu }{4}\|\tilde{\bm{e}}_h^1\|_{\DG}^2 . 
        \end{multline*}
        The above bounds, approximation properties,  the coercivity property of $\aelip$ \eqref{eq:coercivity_a_ellip}, discrete Gronwall's inequality, and the triangle inequality yield bound  \eqref{eq:first_error_estimate_0}. 
    \end{proof}
      Two immediate results of the above arguments are as follows. Under the same regularity assumptions of Theorem \ref{prep:first_error_estimate_velocity} and for $\tau \leq \gamma_1 \mu$ and for $ 1\leq m \leq N_T$, 
    \begin{multline} 
        \sum_{n=1}^m \| \bm{e}_h^{n} - \bm{e}_h^{n-1} \|^2 + \frac{1}{16}\tau^2  \sum_{n=1}^m \vert \phi_h^n \vert_{\DG}^2 + \frac{\kappa \mu}{4} \tau \sum_{n=1}^m  \|\err \|_{\DG}^2 \\ \leq C_\mu \left(1 + \mu+\frac{1}{\mu}\right)( \tau + h^{2k}).  \label{eq:first_error_estimate}
    \end{multline}
    In addition, with \eqref{eq:update_velocity} and \eqref{eq:lift_prop_g}, we have 
\begin{multline} 
\|\bm{v}_h^n - \bm{u}^n \|^2 \leq 2 \|\bm{v}_h^n - \bm{u}_h^n\|^2 + 2 \|\bm{u}_h^n - \bm{u}^n\|^2 \\  \leq C \tau^2 |\phi_h^n|^2_{\DG}  + 2 \|\bm{u}_h^n - \bm{u}^n\|^2 \leq C_\mu \left(1 + \mu+\frac{1}{\mu}\right)( \tau + h^{2k}).
\end{multline}
\begin{remark}
Assume that the mesh is quasi-uniform. If we use \eqref{eq:update_velocity} and the following inverse inequality. $$ \| \bm{\theta}_h\|_{\DG} \leq Ch^{-1} \| \bm{\theta}_h\|, \quad \forall \bm{\theta}_h \in \bm{X}_h, $$
we obtain the following
    \begin{align*}
        \mu \tau \sum_{n=1}^m \| \bm{u}_h^n - \bm{u}^n \|_{\DG}^2 &\leq  \mu \tau \sum_{n=1}^m \|\bm{v}_h^n - \bm{u}^n \|_{\DG}^2 +  \mu \tau \sum_{n=1}^m \tau^2 \| \nabla_h \phi_h^n + \bm{G}_h([\phi_h^n]) \|_{\DG}^2 \nonumber  \\
        & \leq C_\mu \left(1 + \mu+\frac{1}{\mu}\right)( \tau + h^{2k})+ C\mu \tau h^{-2} \sum_{n=1}^m \tau^2 \vert \phi_h^n \vert_{\DG}^2.  
    \end{align*}
Therefore, under a CFL condition, $\tau \leq \gamma_2 h^{2}$, we have the following error bound:
    \begin{align}
     \mu \tau \sum_{n=1}^m \|\bm{u}_h^n - \bm{u}^n \|^2_{\DG} \leq C_\mu \left(1 + \mu+\frac{1}{\mu}\right)( \tau + h^{2k}) . \label{eq:error_estimate_dg_uh}
    \end{align}        
\end{remark}

\section{Numerical experiments}
In this section,  we use the manufactured solution method to compute the convergence rates. Let $\Omega=(0,1)^3$ denote the computational domain and 
let $T = 1$ be the end time for all numerical experiments. The domain $\Omega$ is partitioned into cubic elements. We select the following Beltrami flow with parameter $\mu=1$ as the prescribed solution.
\begin{align*}
u_x(x,y,z,t) &= -\exp{(-t+x)}\sin{(y+z)}-\exp{(-t+z)}\cos{(x+y)},\\
u_y(x,y,z,t) &= -\exp{(-t+y)}\sin{(x+z)}-\exp{(-t+x)}\cos{(y+z)},\\
u_z(x,y,z,t) &= -\exp{(-t+z)}\sin{(x+y)}-\exp{(-t+y)}\cos{(x+z)},\\
p(x,y,z,t)   &= -\exp{(-2t)}\Big(\exp(x+z)\sin{(y+z)}\cos{(x+y)}\\
&\hspace{2.075cm}+\exp(x+y)\sin{(x+z)}\cos{(y+z)}\\
&\hspace{2.075cm}+\exp(y+z)\sin{(x+y)}\cos{(x+z)}\\
&\hspace{2.075cm}+\frac{1}{2}\exp{(2x)} + \frac{1}{2}\exp{(2y)} + \frac{1}{2}\exp{(2z)}
-\bar{p}\Big).
\end{align*}
Here $\bar{p}=7.63958172715414$ is used to force the average pressure over $\Omega$ equal to zero (up to machine precision). The Beltrami flow has the property that the velocity and vorticity vectors are parallel to each other and that the nonlinear convection term is balanced by the pressure gradient, namely, $\partial_t{\bm{u}} - \Delta{\bm{u}} = \bm{0}$ and $\bm{u}\cdot\nabla{\bm{u}} + \nabla{p} = \bm{0}$, which implies $\bm{f} = \bm{0}$. The initial condition and Dirichlet boundary condition on $\partial{\Omega}$ for velocity are imposed by the exact solution.
\par
First, we derive temporal rates of convergence by computing the solutions with time step size $\tau \in \{1/2^3, 1/2^4, ..., 1/2^7\}$. The $\mathbbm{P}2-\mathbbm{P}1$ dG scheme with mesh resolution $h_e = 1/160$ 
is employed to guarantee the error in space does not dominate the error in time. We set $\epsilon = -1$, $\sigma = 64$ on $\Gamma_h$ and $\sigma = 128$ on $\partial\Omega$. The other penalty parameter $\tilde{\sigma}$ is set equal to $2$. If $\mathtt{err}_\tau$ denotes the error when time step size is equal to $\tau$, the convergence rate is defined by $\ln(\mathtt{err}_\tau/\mathtt{err}_{\tau/2})/\ln(2)$. \Cref{tab:time_rate} displays the errors and convergence rates for the velocities and pressure. We observe second order convergence rate in time, which is higher than the expected first order rate.
\par
Next, we obtain spatial rates of convergence by computing the solutions on a sequence of uniformly refined meshes (see the second column of \Cref{tab:space_rate} for $h_e$). We fix $\tau = 1/2^{10}$ for the $\mathbbm{P}1-\mathbbm{P}0$ dG scheme, we fix $\tau = 1/2^{13}$ for the $\mathbbm{P}2-\mathbbm{P}1$ dG scheme, and we fix $\tau = 1/2^{15}$ for the $\mathbbm{P}3-\mathbbm{P}2$ dG scheme to guarantee the spatial error dominates. For $\mathbbm{P}1-\mathbbm{P}0$, we set $\tilde{\sigma} = 1$, $\epsilon = -1$, $\sigma = 8$ on $\Gamma_h$ and $\sigma = 16$ on $\partial\Omega$. For $\mathbbm{P}2-\mathbbm{P}1$, we set $\tilde{\sigma} = 2$, $\epsilon = -1$, $\sigma = 64$ on $\Gamma_h$ and $\sigma = 128$ on $\partial\Omega$. For $\mathbbm{P}3-\mathbbm{P}2$, we set $\tilde{\sigma} = 8$, $\epsilon = -1$, $\sigma = 128$ on $\Gamma_h$ and $\sigma = 256$ on $\partial\Omega$. If  $\mathtt{err}_{h_e}$ is the error on a  mesh with resolution $h_e$, then the rate is defined by $\ln(\mathtt{err}_{h_e}/\mathtt{err}_{h_e/2})/\ln(2)$. We show the errors and rates in \Cref{tab:space_rate}.  The convergence rates are optimal.
\par
In all numerical experiments, the values $\|\bm{v}_h^{N_T}-\bm{u}(T)\|$ and $\|\bm{u}_h^{N_T}-\bm{u}(T)\|$ are identical for at least nine digits.
%
\begin{table}[ht!]
\centering
\begin{tabularx}{\linewidth}{@{~}c@{~}|@{~}c@{~}|C@{\hspace{-0.6666em}}C@{\hspace{-0.6666em}}|C@{\hspace{-0.6666em}}C@{\hspace{-0.6666em}}|C@{\hspace{-0.6666em}}C@{\hspace{-0.6666em}}}
\toprule
$k$ & $\tau$ & $\|\bm{v}_h^{N_T}\!-\bm{u}(T)\|$ & rate & $\|\bm{u}_h^{N_T}\!-\bm{u}(T)\|$ & rate & $\|p_h^{N_T}\!-p(T)\|$ & rate\\
\midrule
$2$ & $1/2^3$ & $1.911$\,E$-2$ & ---     & $1.911$\,E$-2$ & ---     & $2.196$\,E$-1$ & ---    \\
$~$ & $1/2^4$ & $6.437$\,E$-3$ & $1.570$ & $6.437$\,E$-3$ & $1.570$ & $1.025$\,E$-1$ & $1.099$\\
$~$ & $1/2^5$ & $1.699$\,E$-3$ & $1.922$ & $1.699$\,E$-3$ & $1.922$ & $3.093$\,E$-2$ & $1.729$\\
$~$ & $1/2^6$ & $4.506$\,E$-4$ & $1.915$ & $4.506$\,E$-4$ & $1.915$ & $7.790$\,E$-3$ & $1.989$\\
$~$ & $1/2^7$ & $1.207$\,E$-4$ & $1.901$ & $1.207$\,E$-4$ & $1.901$ & $1.919$\,E$-3$ & $2.021$\\
\bottomrule
\end{tabularx}
\caption{Errors and temporal convergence rates of velocity and pressure.}
\label{tab:time_rate}
\end{table}
\begin{table}[ht!]
\centering
\begin{tabularx}{\linewidth}{@{~}c@{~}|@{~}c@{~}|C@{\hspace{-0.6666em}}C@{\hspace{-0.6666em}}|C@{\hspace{-0.6666em}}C@{\hspace{-0.6666em}}|C@{\hspace{-0.6666em}}C@{\hspace{-0.6666em}}}
\toprule
$k$ & $h_e$ & $\|\bm{v}_h^{N_T}\!-\bm{u}(T)\|$ & rate & $\|\bm{u}_h^{N_T}\!-\bm{u}(T)\|$ & rate & $\|p_h^{N_T}\!-p(T)\|$ & rate\\
\midrule
$1$ & $1/2^3$  & $2.849$\,E$-3$ & ---     & $2.849$\,E$-3$ & ---     & $5.255$\,E$-2$ & ---    \\
$~$ & $1/2^4$  & $7.204$\,E$-4$ & $1.984$ & $7.204$\,E$-4$ & $1.984$ & $2.152$\,E$-2$ & $1.288$\\
$~$ & $1/2^5$  & $1.815$\,E$-4$ & $1.989$ & $1.815$\,E$-4$ & $1.989$ & $9.752$\,E$-3$ & $1.142$\\
$~$ & $1/2^6$  & $4.572$\,E$-5$ & $1.989$ & $4.572$\,E$-5$ & $1.989$ & $4.706$\,E$-3$ & $1.051$\\
$~$ & $1/2^7$  & $1.215$\,E$-5$ & $1.912$ & $1.215$\,E$-5$ & $1.912$ & $2.330$\,E$-3$ & $1.014$\\
\midrule
$2$ & $1/2^1$  & $6.322$\,E$-3$ & ---     & $6.322$\,E$-3$ & ---     & $1.925$\,E$-1$ & ---    \\
$~$ & $1/2^2$  & $7.874$\,E$-4$ & $3.005$ & $7.874$\,E$-4$ & $3.005$ & $2.274$\,E$-2$ & $3.082$\\
$~$ & $1/2^3$  & $9.717$\,E$-5$ & $3.019$ & $9.717$\,E$-5$ & $3.019$ & $3.850$\,E$-3$ & $2.562$\\
$~$ & $1/2^4$  & $1.198$\,E$-5$ & $3.020$ & $1.198$\,E$-5$ & $3.020$ & $8.130$\,E$-4$ & $2.244$\\
$~$ & $1/2^5$  & $1.502$\,E$-6$ & $2.996$ & $1.502$\,E$-6$ & $2.996$ & $1.918$\,E$-4$ & $2.084$\\
\midrule
$3$ & $1/2^0$  & $5.129$\,E$-3$ & ---     & $5.129$\,E$-3$ & ---     & $5.021$\,E$-1$ & ---    \\
$~$ & $1/2^1$  & $4.272$\,E$-4$ & $3.586$ & $4.272$\,E$-4$ & $3.586$ & $4.376$\,E$-2$ & $3.521$\\
$~$ & $1/2^2$  & $2.692$\,E$-5$ & $3.988$ & $2.692$\,E$-5$ & $3.988$ & $3.690$\,E$-3$ & $3.568$\\
$~$ & $1/2^3$  & $1.611$\,E$-6$ & $4.063$ & $1.611$\,E$-6$ & $4.063$ & $3.170$\,E$-4$ & $3.541$\\
$~$ & $1/2^4$  & $1.066$\,E$-7$ & $3.918$ & $1.066$\,E$-7$ & $3.918$ & $2.920$\,E$-5$ & $3.441$\\
\bottomrule
\end{tabularx}
\caption{Errors and spatial convergence rates of velocity and pressure.}
\label{tab:space_rate}
\end{table}

\section{Conclusion} \label{sec:conclusions}
The main contribution of this work is a theoretical error analysis of a dG pressure correction scheme for solving the incompressible Navier-Stokes equations. We establish stability of the method
and derive a priori error estimates for the discrete velocity that are optimal in the broken gradient norm. The convergence analysis is technical and relies on lift operators and appropriate local approximations. Future work will consist of deriving error bounds for the pressure and 
optimal error estimates for the discrete velocity in the L$^2$ norm.
\bibliographystyle{amsplain}
\bibliography{references}

\end{document}